\documentclass{ws-igtr}
\usepackage{graphicx, mathtools, doi, tikz, pgfplots, adjustbox, booktabs, subcaption, cleveref,setspace}
\pgfplotsset{compat=1.14}
\graphicspath{{Figures/}{./}}
\newcommand{\ex}{\mathbb{E}}
\newcommand{\m}{\mathrm{m}}

\newcommand{\us}{u^s\(r\)}
\newcommand{\fdr}{\frac{\mathop{du^s}}{\mathop{dr}}\(r\)}
\renewcommand{\(}{\left(}
\renewcommand{\)}{\right)}
\DeclareMathOperator*{\argmax}{arg\,max}

\begin{document}

\markboth{Stefanos Leonardos and Costis Melolidakis}
{Endogenizing the cost parameter in Cournot oligopoly}

\catchline{}{}{}{}{}

\title{ENDOGENIZING THE COST PARAMETER IN COURNOT OLIGOPOLY\footnote{
This research has been supported by the Research Funding Program ARISTEIA II: "Optimization of stochastic systems under partial information and applications" of NSRF.}}
\author{Stefanos Leonardos\,\, and Costis Melolidakis}
\address{Department of Mathematics, National \& Kapodistrian University of Athens,\\ 
Panepistimioupolis GR - 157 84, Athens, Greece\,\\ 
\email{sleonardos@math.uoa.gr, cmelol@math.uoa.gr}}

\maketitle

\begin{history}
\received{(Day Month Year)}
\revised{(Day Month Year)}
\end{history}

\begin{abstract}
We study the effects of endogenous cost formation in the classic Cournot oligopoly through an extended two-stage game. The competing Cournot firms produce low-cost but limited quantities of a single homogeneous product. For additional procurements, they may refer to a revenue-maximizing supplier who sets a wholesale price prior to their orders. We express this chain as a two-stage game and study its equilibrium under two different information levels: complete and incomplete information on the side of the supplier about the actual market demand. In the deterministic case, we derive the unique subgame perfect Nash equilibrium for different values of the retailers' capacity levels, supplier's cost and market demand. To study the incomplete information case, we model demand uncertainty via a continuous probability distribution. Under mild assumptions, we characterize the supplier's optimal pricing policy as a fixed point of a proper translation of his expectation about the orders that he will receive from the retailers. If this expectation is decreasing in his price, then such an optimal policy always exists and is unique. Based on this characterization, we are able to proceed with comparative statics and sensitivity analysis, both analytically and numerically. Incomplete information gives rise to market inefficiencies because the supplier may ask for a too high price. Increasing supplier's cost results in increasing wholesale prices, decreasing orders from the retailers and hence decreasing consumer surplus. Increasing retailers' production capacities results in decreasing wholesale prices and increasing consumer surplus. Finally, as the number of second-stage retailers increases, the supplier's profit may initially rise but eventually drops.
\end{abstract}
\section{Introduction}\label{secintroduction}
Modern oligopolistic firms build up their retail stock from various sources that may include in-house production or procurements from large, internationally operating manufacturers. These suppliers also act strategically by setting wholesale prices to maximize their revenues. Hence, the robustness and validity of economic intuitions that have been obtained under classic assumptions of cost rigidity as in the standard Cournot quantity-competition model need to be revised under more realistic, yet mathematically tractable cost structures.\par
The main objective of this paper is to provide the proper game-theoretic framework to study the cost formation as a strategic decision in the standard Cournot oligopoly. To endogenize the cost, we extend the classic model in a two-stage game. In the second stage, the competing firms produce limited quantities of the product up to a specified capacity. For additional procurements, they refer to an external supplier, who has ample capacity but may be uncertain about the actual demand that the retailers are facing. The supplier acts as a Stackelberg leader and sets the wholesale price in the first stage, prior to the decision of the retailers. We study a complete and incomplete information model and apply the subgame-perfect Nash and Bayes-Nash equilibrium concepts, respectively, to analyze the strategic decisions of the market participants and gain economic intuition. \par
In a closely motivated study, \cite{Ma15} point out that the problem of strategic cost formation has not yet been appropriately addressed in the quite extensive Cournot literature. The classic model's assumption of constant, exogenously given marginal cost does not reflect the complex cost structures of contemporary economic practice. To address this issue, they study the game that arises when competing Cournot firms purchase their inputs from a common supplier. They examine contracts, with the competing firms having the bargaining power, and hence, their analysis departs from the present study. \cite{Ca04} highlight the need for a game theoretical analysis of more dynamic settings in oligopolistic markets. They report only scarce applications of the subgame-perfect equilibrium, all in settings quite different from ours. Since then, several papers have appeared in the relevant literature. \cite{Es09} apply the Stackelberg-strategy solution concept to study the interaction between a seller and a buyer. However, their interest shifts to investigating the impact of marketing (advertising) expenditures. \cite{Zh10} study wholesale pricing schemes between manufacturers and retailers, but their focus is on cooperative mechanisms that result in inventory coordination and supply chain improvement. \par
The present paper focuses on the relationship between market demand and various costs as the key in studying the effects of an exogenous source of supply for Cournot oligopolists. If the demand is high enough, the competing firms may have incentive to place orders with an external supplier, depending, of course, on the price he asks. In such a study, various questions have to be addressed: Do the firms have production capacities of their own? If yes, then these capacities have to be assumed bounded (as they are in reality) so that need may arise for additional procurements. Does the supplier know the actual demand the oligopolists face? If no, then he may ask for a price that the competing firms will not accept, even if it is to his advantage not to do so. The latter question implies an incomplete information game setup. By viewing the interaction of the competing Cournot oligopolists with their supplier as a two-stage game, the cost parameter of the classic Cournot model is endogenized, and answers can be worked out.

\subsection{Model summary}
In detail, inspired by the classic Cournot oligopoly in which producers/retailers compete over quantity, we study the market of a homogeneous good differing from the classic model in the following ways. Each producer/retailer may produce only a limited quantity of the good up to a specified and commonly known production capacity. If needed, the retailers may refer to a single supplier to order additional quantities. The supplier may produce unlimited quantities but at a higher cost than the retailers, making it best for the retailers to exhaust their production capacities before placing additional orders. The market clears at a price that is determined by an affine inverse demand function. The demand parameter or demand intercept is considered a random variable with a commonly known non-atomic, i.e., continuous, distribution having a finite expectation. \par
Depending on the time of the demand realization, two variations of this market structure are examined, corresponding to a complete and an incomplete information two-stage game. In the first stage, the supplier fixes a price by deciding his profit margin and then, in the second stage, the retailers, knowing this price and the true value of the demand parameter, decide their production quantity -- up to capacity -- and place their orders, if any. The decisions of the producers/retailers are simultaneous. The supplier may -- complete information -- or may not -- incomplete information -- know the true value of the demand parameter before making his decision. In the second case, when the demand is realized after the supplier has set his price, the demand parameter or equivalently the supplier's belief about it is modelled by a continuous probability distribution. In both variations, demand uncertainty is resolved prior to retailers' decision. In this respect, our setup differs from the usual incomplete information models of Cournot markets, in which demand uncertainty involves the producers, e.g., see \cite{Ei10} and references given therein. \par 
The limited production capacities of the producers/retailers may equivalently be interpreted as inventory quantities that are drawn at a fixed cost per unit. If the production capacities are $0$, this is a classic Cournot model with the cost input determined exogenously. This case exhibits independent interest and has been exhaustively treated in a companion work \cite{Le17}.\footnote{Other related results include \cite{Le18,Leo20,Leo21}, preliminary versions of which appear in \cite{Bel18,Ko18}.}. The present market model aims to describe a wide range of composite cost structures that are encountered in contemporary practice. The energy market, in which domestic firms produce up to limited capacities and refer to international suppliers in periods of high demand, agriculture, apparel, food products and numerous more examples can be thought of as real economic applications. In all these cases, the suppliers have ample quantities to cover any domestic demand. Hence, they only need to decide on the wholesale price that they will ask. However, due to their distance from the retail market, they may have limited information about the actual retail demand, an aspect that we implement by treating the demand parameter as a random variable on the side of the supplier. The realization of the demand after the pricing decision of the supplier is the device to implement this uncertainty. The model assumptions are further discussed in \Cref{discussion}.

\subsection{Overview of results}
For a more transparent exposition, we first analyze the case of $n=2$ competing retailers (duopoly) and then generalize our results to the case of arbitrary $n>2$ competing retailers (oligopoly). We study the equilibrium behavior of the two-stage game by first determining the unique equilibrium solution of the second-stage game, which concerns the quantities that the producers/retailers will produce and the additional quantities they will order from the supplier. The equilibrium strategies are given in \Cref{prop3.3full} and \Cref{fig:3}. The second stage is the same in both the complete and the incomplete information case. Then, we examine the first-stage game, which concerns the price the supplier will ask for the product. In the case of the complete information duopoly, a unique subgame-perfect Nash equilibrium always exists, which we explicitly determine in \Cref{prop3.4full} for all values of production capacities $T_1,T_2$, supplier's cost $c$ and demand level $\alpha$. While equilibrium strategies, payoffs of the retailers, and equilibrium payoff of the supplier are continuous and increasing in the demand level $\alpha$, the supplier's optimal pricing strategy may not be continuous and even not monotonic in $\alpha$. Under the assumption of symmetric (identical) retailers, i.e., $T_1=T_2=T$, the complete information case simplifies significantly. The resulting equilibrium is given in \Cref{thm3.5} and allows for a discussion on the impact of the retailers' capacities on the consumer surplus. \par
Our focus is mainly on the incomplete information case. In \Cref{thm4.1}, we show that if a subgame perfect Bayesian-Nash equilibrium exists, then it will necessarily be a fixed point of a translation of the Mean Residual Lifetime (MRL) function of the supplier's belief. If the support of the supplier's belief is bounded, then such an equilibrium exists. If the MRL function is decreasing (as in most ``well behaved'' distributions), then a subgame perfect equilibrium always exists and is unique, irrespective of the supplier's belief support. The fixed point characterization of the supplier's equilibrium strategy provides a powerful tool for a transparent comparative statics and sensitivity analysis. \par
Accordingly, in \Cref{secinef}, we study market inefficiencies due to incomplete information. Through analytical considerations and numerical simulations, we reason about the supplier's incentives to charge a too high price, despite running a considerable risk of no transactions between him and the retailers. In \Cref{inventory}, we observe that the supplier's profit margin and wholesale price both exhibit a reverse monotone relationship to the retailers' inventory level. Similarly, as the supplier's cost increases, his profit margin decreases. However, in this case, the wholesale price that he asks increases and the total quantity ordered by the retailers decreases, see \Cref{cost}. Consequently, the total quantity that is sold to the consumers decreases, which has a negative impact on consumer surplus. We confirm these observations numerically. Finally, in \Cref{thmnretailers} we show that \Cref{thm4.1} generalizes to an arbitrary number of second-stage retailers which enables the study of their impact on the supplier's profit. Numerical simulations with specific distributions reveal that while the supplier may initially benefit from an increasing number of retailers, eventually his profits will drop as their number, and consequently their total in-house production capacities, continue to increase. Accordingly, our main contributions can be summarized in the following points
\begin{itemlist}
\item formulation of a game-theoretic framework to incorporate the complex cost structure of modern oligopolistic Cournot-firms.
\item closed form characterization of the equilibria in both complete and incomplete information market settings; 
\item utilization of the MRL function, which, although useful to study stochastic payoff functions, has received limited attention in the revenue management literature; and 
\item comparative statics and sensitivity analysis of the basic model parameters -- facilitated by the previous characterizations -- to understand the economic implications of the present model, both analytically and numerically.
\end{itemlist}
Still, the complexity and variety of real economic models provide numerous possibilities for future research and extensions. However, the proposed model aims to provide a benchmark for related studies, due to its mathematical tractability that is mainly highlighted by the closed form characterizations of the equilibrium strategies both in the complete and incomplete information settings. For a more detailed discussion of the economic assumptions and possible extensions of the present model, see \Cref{discussion} and \Cref{secfuture}.

\subsection{Related Work}
The current work combines elements from various active research areas. Cournot competition with an external supplier having uncertainty about the retail demand may be viewed as an application of game theoretic tools in supply chain management. However, our interest is in the equilibrium behavior of the complete and incomplete information two-stage game rather than in building coordination mechanisms. The incomplete information of the supplier relates our work to incomplete information Cournot models, although the lack of information refers to the producers in most papers in this area, see e.g., \cite{Be05} or \cite{Wu16}. \par
Price-only contracts between suppliers and retailers, like the one described here, continue to attract widespread attention in the literature. \cite{La01} argue that the simplicity of price-only contracts makes them particularly attractive if they do not significantly reduce supply-chain efficiency. Since unmodeled factors are unlikely to reverse the insights derived by such models, they conclude that as such, they provide (at least in many cases) a benchmark, worst-case analysis. \cite{Pe07} highlight the practical prevalence of price-only contracts and study the efficiency of various decentralized supply chains that use price-only contracts by measuring the respective Price of Anarchy of the chain. In a similar fashion to our \Cref{thmbound}, they obtain bounds on the market fashion by restricting attention to the IGFR class of random variables. \par 
\cite{Ad15} study a supply chain with a single supplier and symmetric retailers analytically and find that the supplier prefers to have as many retailers as possible in the market, even in several restrictive cases. This is in contrast to our findings in \Cref{sec7}, which shows that such conclusions may be model-specific. In a similar model, \cite{Wu12} investigate the pricing decisions in a non-cooperative supply chain that consists of two retailers and one common supplier. As the supplier's decision variables, they consider the wholesale prices to the retailers, and argue that under active market regulations (Robinson-Patman Act) the supplier is required to charge a uniform price -- and hence not price-differentiate between the retailers -- when the retailers place their orders simultaneously, as in our model. In their approach however, each retailer's decision is the sale price to the market, and hence their analysis differs from ours.\par
Having capacity constraints for the producers/consumers has also attracted attention in the Cournot literature, see e.g., \cite{Bi09}. \cite{Be05} investigate the equilibrium behavior of a decentralized supply chain with competing retailers under demand uncertainty. Their model accounts for demand uncertainty from the retailers' point of view as well and focuses mainly on the design of contracts that will coordinate the supply chain. Based on whether the uncertainty of the demand is resolved prior to or after the retailers' decisions, they identify two main streams of literature. For the first stream, in which the present paper may be placed, they refer to \cite{Vi01} for an extensive survey. \par
However, capacity constrained duopolies are mostly studied in view of price rather than quantity competition. \cite{Os86}, and the references therein are among the classics in this field. Equally common is the study of models in which the capacity constraints are viewed as inventories kept by the retailers at a lower cost. Papers in this direction focus mainly on determining optimal policies in building the inventory over more than one period. \cite{Ha15} and the references therein are indicative of this field of research.\par
\cite{Ei10} examine Cournot competition under incomplete producers' information about demand and production costs. They provide examples of such games without a Bayesian Cournot equilibrium in pure strategies, discuss the implication of not allowing negative prices and provide additional sufficient conditions that will guarantee existence and uniqueness of equilibrium. \cite{Ri13} discusses Cournot competition under incomplete producers' information about their production capacities and proves the existence of equilibrium under the assumption of stochastic independence of the unknown capacities. He also discusses simplifications of the inverse demand function that result in symmetric equilibria and implications of information sharing among the producers. \par

\subsection{Outline}
The rest of the paper is structured as follows. In Sections 2--6, we discuss a market with one supplier and two producers/retailers. In particular, in \Cref{secmodel}, we build up the formal setting for the market model. In \Cref{sec3}, we present some preliminary results and treat the complete information case. In \Cref{sec4}, we analyze the model of incomplete information. In \Cref{secinef}, we compare the complete and incomplete information equilibrium outcomes and study both analytically and numerically the inefficiencies that occur in the incomplete information setting. In \Cref{seccost}, we perform sensitivity analysis on the model parameters and perform numerical simulations to illustrate the results. Finally, in \Cref{sec7}, we generalize \eqref{thm4.1} to the case of $n>2$ identical retailers and study their impact on the supplier's profits. \par 
To focus on the economic and game-theoretic aspects of the analysis, all proofs are presented in \ref{app1}. While proving the statements for the complete information case is by exhaustive case discrimination, the proofs concerning the incomplete information case utilize probabilistic results from the theory of the mean residual life (MRL) function and may be of independent interest. The MRL function is widely used in reliability analysis, but its formal applications in economics are still scarce.\par
Throughout the rest of this paper we drop the double name ``producers/retailers'' and use only the term ``retailers''.

\section{The Model}\label{secmodel}
We consider the market of a homogenous good that consists of two producers/retailers ($R_1$ and $R_2$) who compete over quantity ($R_i$ places quantity $Q_i$), and a supplier (or wholesaler) under the following assumptions.  
\begin{arabiclist}
\item The retailers may produce quantities $t_1$ and $t_2$ up to a capacities $T_1$ and $T_2$, respectively, at a common fixed cost $h$ per unit normalized to zero\footnote{See \Cref{secintroduction} for an alternative interpretation of these quantities as inventories.}. 
\item Additionally, they may order quantities $q_1, q_2$ from the supplier at a price $w$ set prior to and independently of their orders. The total quantity $Q_i\(w\), i=1,2$ that each retailer releases to the market is equal to the sum 
\begin{equation*}Q_i\(w\):=t_i\(w\)+q_i\(w\) \end{equation*}
or shortly $Q_i:=t_i+q_i$, where, for $i=1,2$, the variable $t_i\le T_i$ is the quantity that retailer $R_i$ produces by himself or draws from his inventory (at normalized zero cost) and $q_i$ is the quantity that the retailer $R_i$ orders from the supplier at price $w$. 
\item The supplier may produce unlimited quantities of the good at a cost $c$ per unit. We assume that the retailers are more efficient in the production of the good or equivalently that $c>h$. After the normalization of the retailers' production cost $h$ to $0$, the rest of the parameters, i.e., $w$ and $c$ are also normalized. Thus, $w$ represents a normalized price, i.e., the initial price that was set by the supplier minus the retailers' production cost and $c$ a normalized cost, i.e., the supplier's initial cost minus the retailers' production cost. The supplier's profit margin $r$ is not affected by the normalization and is equal to 
\begin{equation*} r:=w-c \end{equation*} 
\item After the retailers set the total quantity $Q=Q_1+Q_2$ that will be released to the market, the market clears at a price $p$ that is determined by an inverse demand function, which we assume to be affine\footnote{After normalization of the slope parameter (initially denoted with $\beta$) of the inverse demand function to $1$, all variables in \eqref{demand} are expressed in the units of quantity and not in monetary units and therefore any interpretations or comparisons of the subsequent results should be done with caution.}
\begin{equation}\label{demand} p=\alpha-Q \end{equation} 
\item The demand parameter $\alpha$ is a non-negative random variable with finite expectation $\ex \(\alpha\)<+\infty$ and a continuous cumulative distribution function (cdf) $F$ (i.e., the measure induced on the space of $\alpha$ is non-atomic). We will assume that $\alpha \ge h$ for all values of $\alpha$, i.e., that the demand parameter is greater or equal to the retailers' production cost. The latter assumption is consistent with the classic Cournot duopoly model, which is  resembled by the second stage of the game (however, the second-stage game is not a classic Cournot duopoly due to the capacity constraints $T_1$ and $T_2$ and to the possibility of $w>\alpha$). After normalization, in what follows, we will use the term $\alpha_H$ (resp. $\alpha_L$) to denote the lower upper bound (respectively the upper lower bound) of the support of $\alpha$.
\item The capacities $T_1, T_2$ and the distribution of the random demand parameter $\alpha$ are common knowledge among the three participants of the market (the retailers and the supplier).
\end{arabiclist}
Based on these assumptions, a strategy $s_i\(w\)$ for retailer $R_i, i=1,2$ is a vector valued function $s_i(w)=(t_i (w), q_i (w))$ or shortly a pair $s_i=(t_i, q_i)$ for $i=1,2$. Equation \eqref{demand} implies that $Q_i$ may not exceed $\alpha$ and hence the strategy set $\tilde{S}^i$ of $R_i$ will satisfy
\begin{equation}\label{strategy} \tilde{S}^i \subset \left\{(t_i,q_i) : 0\le t_i \le T_i \text{ and } 0\le q_i+t_i \le \alpha\right\} \end{equation}
Denoting by $s:=\(s_1, s_2\)$ a strategy profile, the payoff $u^i\(s \mid {w}\)$ of retailer $R_i, i=1,2$ will be given by 
\begin{equation}\label{3} u^i\(s\mid w\)= Q_i\(\alpha-Q\)-w  q_i=Q_i\(\alpha-w-Q\)+w  t_i \end{equation}
Whenever confusion may not arise, we will write $u^i(s)$ instead of $u^i(s\mid w)$.
A strategy for the supplier is the price $w$ he charges to the retailers or equivalently his profit margin $r$. From \eqref{3}, we see that $w$ may not exceed $a_H$, because otherwise the retailers will not order. Additionally, it may not be lower than $c$, since in that case, his payoff will become negative. Hence, in terms of his profit margin $r$, the strategy set $R$ of the supplier satisfies 
\begin{equation}\label{4}R\subset \{r : 0 \le r \le \alpha_H-c\} \end{equation}
Consequently, a reasonable assumption is that $c<\alpha_H$, because otherwise, the problem becomes trivial from the supplier's perspective.
For a given value of $\alpha$, the supplier's payoff function, stated in terms of $r$ rather than $w=r+c$, is given by 
\begin{equation}\label{5}u^s\(r\mid \alpha\)=r\(q_1\(w\)+q_2\(w\)\) \end{equation}
On the other hand, it is not necessary for the retailers to know the exact values of $c$ and $r$, and hence, from their point of view (second-stage game), we keep the notation $w=r+c$. If the supplier does not know $\alpha$ (incomplete information case), then his payoff function will be 
\begin{equation}\label{6}\us:=\ex u^s\(r\mid \alpha\) \end{equation}
where the expectation is taken with respect to the distribution of $\alpha$. 

\subsection{Two market models with different information structure}
To proceed with the formal two-stage game model, we recall that both the production/inventory capacities $T_1$ and $T_2$ of the retailers and the distribution of the demand parameter $\alpha$ are common knowledge to the three market participants. We then have, 
\begin{itemlist}
\item \textbf{Complete Information Case:} The demand parameter $\alpha$ is realized and observed by both the supplier and the retailers\footnote{Of course, this means that there is no randomness in $\alpha$, so the description of $\alpha$ as a random variable is redundant in the complete information case. We use it just to give a common formal description of both the complete and the incomplete information case.}. Then, at stage 1, the supplier fixes his profit margin $r$ and hence his price $w$. His strategy set and payoff function are given by \eqref{4} and \eqref{5}. At stage 2, based on the value of $w$, each competing retailer chooses the quantity $Q_i, i=1,2$ that he will release to the market by determining how much quantity $t_i$ he will draw (at zero cost) from his inventory $T_i$ and how much additional quantity $q_i$ he will order (at price $w$) from the supplier. The strategy sets and payoff functions of the retailers are given by \eqref{strategy} and \eqref{3}, respectively.
\item\textbf{Incomplete Information Case:} At stage 1, the supplier chooses $r$ without knowing the true value of $\alpha$. His strategy set remains the same, but his payoff function is now given by \eqref{6}. After $r$, and hence $w=r+c$, are fixed, the demand parameter $\alpha$ is realized, and along with the price $w$ is observed by the retailers. Then we proceed to stage 2, which is identical to that of the Complete Information Case. 
\end{itemlist}
All the above are assumed to be common knowledge of the three players.

\section{Subgame-perfect equilibria under complete information}\label{sec3}
First, we treat the case with no uncertainty on the side of the supplier about the demand parameter $\alpha$. The subgame perfect equilibria of this two-stage game are determined in \Cref{sub3.1,sub3.2}. As is intuitively expected, it is best for the retailers to produce up to their capacity constraints or equivalently to exhaust their inventories before ordering additional quantities from the supplier at unit price $w$. For simplicity in the notation of \Cref{lem3.1,lem3.2}, fix $i \in \{1,2\}$ and let $T_i:=T$. As above, $Q_i=t_i+q_i$.
\begin{lemma}\label{lem3.1}
Any strategy $s_i=\(t_i, q_i\) \in \tilde{S}^i$ with $t_i<T$ and $q_i>0$ is strictly dominated by a strategy $s_i'=\(t_i',q_i'\)$ with \begin{equation*} \(t_i',q_i'\)=\begin{cases}\(T, Q_i-T\), & \text{ if }Q_i\ge T \\ \(Q_i, 0\), & \text{ if }Q_i< T \end{cases}\end{equation*} 
or equivalently by $\(t_i',q_i'\)=\(\min{\{Q_i,T\}}, \(Q_i-T\)^+\)$.
\end{lemma} 
\begin{proof} Let $Q_i\ge T$. Then for any $s_j \in \tilde{S}^j$, \eqref{3} implies that $u^i\(s_i, s_j\)=u^i\(s'_i,s_j\)+w\(t_i-T\)$. Since $t_i<T$, the result follows. Similarly, if $Q_i< T$, then for any $s_j \in \tilde{S}^j$, \eqref{3} implies that
$u^i\(s_i, s_j\)=u^i\(s'_i,s_j\)-wq_i.$
\end{proof} 
\noindent Accordingly, we restrict attention to the strategies in 
\begin{align*}S^i=&\bigg\{\(t_i,q_i\) : 0\le t_i < \min{\{\alpha, T\}}, q_i=0\bigg\}\bigcup \\& 
\bigcup\bigg\{\(t_i,q_i\): T=\min{\{\alpha, T\}}, 0\le q_i\le \(\alpha-T\)^+\bigg\} \end{align*}
\Cref{fig:1a} depicts the set $S^i$ when $T<\alpha$ and \Cref{fig:1b} when $T\ge\alpha$.
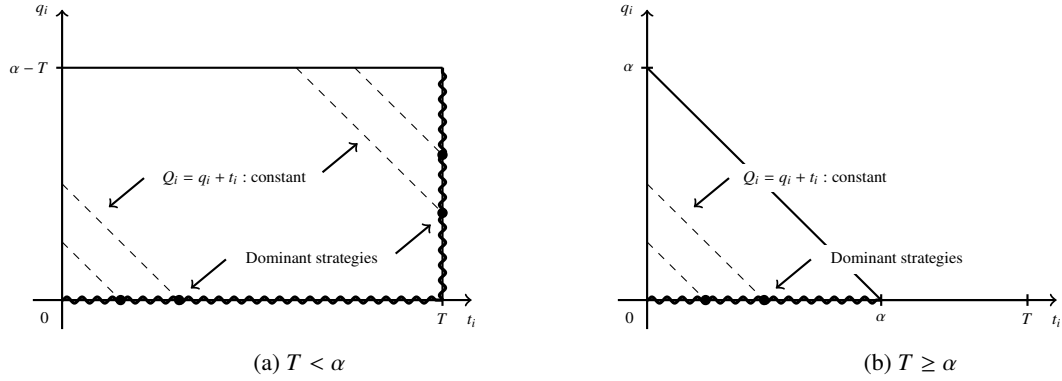
\begin{figure}[!htp]\centering
\begin{subfigure}{.49\textwidth}
\begin{tikzpicture}[domain=0:8,thick,scale=0.77, every node/.style={transform shape}] 
\draw[->](-0.5,0) -- (7,0);
\draw[->](0,-0.5) -- (0,5);
\draw[thick] (7,-0.13) node[below]{$t_i$}(-0.1,5) node[left]{$q_i$} (6.5,4)--(6.5,-0.1) node[below]{$T$} (6.5,0)--(0,0)
(-0.1,4)  node[left]{$\alpha-T$} --(6.5,4) (-0.3, -0.3) node{$0$};
\draw[very thick, decorate,decoration={snake,amplitude=.4mm,segment length=2mm}](0,0)--(6.5,0); 
\draw[very thick, decorate,decoration={snake,amplitude=.4mm,segment length=2mm,post length=0mm}](6.5,0)--(6.5,4);
\draw[->] (1.4,2.1) node[above, right]{\hspace{3pt} $Q_i=q_i+t_i:$ constant}--(0.8,1.6);
\draw[->] (4.45,2.2)--(5.05,2.7);
\draw[->] (2.8, 0.7) node[above, right]{\hspace{6pt}Dominant strategies}--(2.2, 0.2);
\draw[->] (5.7, 0.8)--(6.3, 1.3);
\draw[dashed, thin] (0,1)--(1,0) (0,2)--(2,0) (5,4)--(6.5,2.5) (4, 4)--(6.5,1.5);
\filldraw (2,0) circle (2pt) (1,0) circle (2pt) (6.5,2.5) circle (2pt) (6.5,1.5) circle (2pt);
\end{tikzpicture}
\caption{$T<\alpha$}
\label{fig:1a}
\end{subfigure}%
\hspace{0.15cm}
\begin{subfigure}{0.49\textwidth}
\begin{tikzpicture}[domain=0:8, thick, scale=0.77, every node/.style={transform shape}] 
\draw[->](-0.5,0) -- (7,0);
\draw[->](0,-0.5) -- (0,5);
\draw[thick](-0.1,5) node[left]{$q_i$} (4,0.1)--(4,-0.1) node[below]{$\alpha$} (6.5,0)--(0,0) (-0.1,4)  node[left]{$\alpha$} --(0.1,4) (0,4)--(4,0) (6.5,0.1)--(6.5,-0.1) node[below]{ $T$} (-0.3, -0.3) node{$0$} (7,-0.13) node[below]{$t_i$};
\draw[very thick, decorate,decoration={snake,amplitude=.4mm,segment length=2mm}](0,0)--(4,0); 
\draw[->](1.4,2.1) node[above, right, fill=white]{\hspace{3pt}$Q_i=q_i+t_i:$ constant}--(0.8,1.6);
\draw[->](2.8, 0.7) node[above, right, fill=white]{\hspace{6pt}Dominant strategies}--(2.2, 0.2);
\draw[dashed, thin](0,1)--(1,0) (0,2)--(2,0);
\filldraw (2,0) circle (2pt) (1,0) circle (2pt);
\end{tikzpicture}
\caption{$T\ge \alpha$}
\label{fig:1b}
\end{subfigure}
\caption{Retailers strategy set $S^i$.}
\label{fig:1}\end{figure}

As shown below, Lemma \ref{lem3.1} considerably simplifies the maximization of $u^i\(\cdot, s_j\)$, since for any strategy $s_j$ of retailer $R_j$, the maximum of $u^i\(\cdot, s_j\)$ will be attained at the bottom or right-hand side boundary of the region $[0,\min{\{\alpha,T\}}]\times\left[0,\(\alpha-T\)^+\right]$. Moreover, Lemma \ref{lem3.1} implies that when $\alpha\le T$, retailer $R_i$ will order no additional quantity from the supplier\footnote{In Proposition \ref{prop3.3} we will see that in equilibirum retailer $R_i$ will order no additional quantity if $\alpha\le 3T$.}. Although trivial, we may not exclude this case in general, since in \Cref{sec4}, we consider $\alpha$ to be varying. 

\subsection{Unique second-stage equilibrium strategies}\label{sub3.1}
Restricting attention to $S^i$ for $i=1,2$, we obtain the best reply correspondences $\operatorname{BR^1}$ and $\operatorname{BR^2}$ of retailers $R_1$ and $R_2$, respectively. To proceed, we notice that the payoff of retailer $R_i, i=1,2$ depends on the total quantity $Q_j$ that retailer $R_j, j=3-i$ releases to the market and not on the explicit values of $t_j, q_j$, cf. \eqref{3}.
\begin{lemma} \label{lem3.2}
The best reply correspondence $\operatorname{BR}^i\(Q_j\)=\(t_i, q_i\)$ of retailer $R_i$ for $i=1,2$ is given by 
\[\operatorname{BR}^i(Q_j)=\left\{\begin{array}{lllrr}
\(T,\frac{\alpha-w-Q_j}{2}-T\), &\,\,\,&\text{ if }\,\, 0\le Q_j< \alpha-w-2T&& (1)\\ 
\(T, 0\), &&\text{ if }\,\, \alpha-w-2T\le Q_j< \alpha-2T&&(2)\\
\(\frac{\alpha-Q_j}{2}, 0\), & &\text{ if }\,\, \alpha-2T\le Q_j &&(3)
\end{array} \right.\]
\end{lemma}
\begin{proof} See \ref{app1}. Enumeration $\(1\), \(2\), \(3\)$ of the different parts of the best reply correspondence will be used for a more clear case discrimination in the subsequent equilibrium analysis. \end{proof}
A generic graph of the best reply correspondence $\operatorname{BR^1}$ of retailer $R_1$ to the total quantity $Q_2$ that retailer $R_2$ releases to the market is given in \Cref{fig:2}.
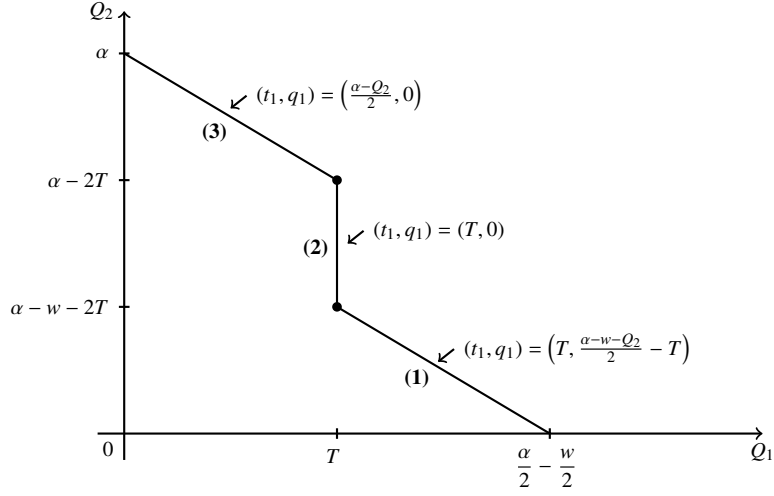
\begin{figure}[!htp]\centering
\begin{tikzpicture}[domain=0:12,thick, scale=0.7]
\filldraw (4,2.4) circle (2pt) (4,4.8) circle (2pt); 
\draw[->](-0.5,0) -- (12,0) node[below] {$Q_1$};
\draw[->](0,-0.5) -- (0,8) node[left] {$Q_2$};
\draw (4,0.1)--(4,-0.1) node[below]{$T\,\,$} (8,0.1)--(8,-0.1) node[below]{$\dfrac{\alpha}{2}-\dfrac{w}{2}\,\,$} 
(0.1,2.4)--(-0.1,2.4) node[left]{$\alpha-w-2T$} (0.1,4.8)--(-0.1,4.8) node[left]{$\alpha-2T$} (0.1,7.2)--(-0.1,7.2) node[left]{$\alpha$}
(0,7.2)--(4,4.8)--(4,2.4)--(8,0) (5.5,1.05) node {\textbf{(1)}} (3.6,3.5) node {\textbf{(2)}} (1.7,5.7) node {\textbf{(3)}} (-0.3, -0.3) node{$0$};
\draw[->](6.2,1.6) node[right]{$(t_1,q_1)=\(T,\frac{\alpha-w-Q_2}{2}-T\)$}--(5.9,1.36);
\draw[->](4.5,3.84) node[right]{$(t_1,q_1)=(T,0)$}--(4.2,3.6);
\draw[->](2.3,6.4) node[right]{$(t_1,q_1)=\(\frac{\alpha-Q_2}{2},0\)$}--(2,6.16);
\end{tikzpicture}
\caption{Best reply correspondence of Retailer $R_1$.}
\label{fig:2} \end{figure}
The equilibrium analysis of the second stage of the game proceeds in the standard way, i.e., with the identification of Nash equilibria through the intersection of the best reply correspondences $\operatorname{BR^1}$ and $\operatorname{BR^2}$. From the explicit form of the best reply correspondence that is given in Lemma \ref{lem3.2} (see also \Cref{fig:2}), it is straightforward that the equilibrium strategies depend on the values of $\alpha$ and $w$ and their relation to $T_1, T_2$. For convenience, we will denote with $\Gamma_{ij}$ the case that the equilibrium occurs as an intersection of parts $i,j$ for $i, j=1,2,3$. Since by assumption $T_1\ge T_2$, only the cases with $i\ge j$ (instead of all possible $9$ cases) may occur. 
\begin{proposition}\label{prop3.3full}
Given the values of $\alpha$ and $w$, the second-stage equilibrium strategies between retailers $R_1$ and $R_2$ for all possible values of $T_1 \ge T_2$ are given by
\begin{center}\begin{adjustbox}{max size={\textwidth}{\textheight}}
\begin{tabular}{c|l|l|l|l|l}
\midrule
\multicolumn{2}{l|}{}&\multicolumn{4}{c}{\textbf{Equilibrium Strategies}}\\
\cmidrule{3-6}\textbf{Case} & \textbf{Range of $\alpha$} &$t_1^*$ & $q_1^*$ & $t_2^*$ & $q_2^*$\\ \toprule
$\Gamma_{11}$ & $(3T_1+w, \infty)$ & $T_1$ & $\frac{\alpha-w}{3}-T_1$ & $T_2$ & $\frac{\alpha-w}{3}-T_2$\\
$\Gamma_{21}$ & $\(\max\left\{3T_1-w, T_1+2T_2+w\right\}, 3T_1+w\right]$ & $T_1$ & $0$ & $T_2$ & $\frac{\alpha-w-T_1}{2}-T_2$ \\
$\Gamma_{22}$ & $(2T_1+T_2, T_1+2T_2+w]$ & $T_1$ & $0$ & $T_2$ & $0$\\
$\Gamma_{31}$ & $(3T_2+2w, 3T_1-w]$ & $\frac{\alpha+w}{3}$ & $0$ & $T_2$ & $\frac{\alpha-2w}{3}-T_2$\\
$\Gamma_{32}$ & $\left(3T_2, \min\left\{2T_1+T_2, 3T_2+2w\right\}\right]$ & $\frac{\alpha-T_2}{2}$ & $0$ & $T_2$ & $0$ \\
$\Gamma_{33}$ & $[0, 3T_2]$ & $\frac{\alpha}{3}$ & $0$ & $\frac{\alpha}{3}$ & $0$ \\ \midrule
\end{tabular}\end{adjustbox}
\captionsetup{type=table} \captionof{table}{Second stage equilibrium strategies for all $T_1 \ge T_2$.}
\label{tab:1}\end{center}
\end{proposition}
The second-stage equilibria regions in the $T_1-T_2$ plane are depicted in \Cref{fig:3}. One should notice that the conditions under which cases $\Gamma_{ij}$ (second column of \Cref{tab:1}) apply are mutually exclusive, and hence, there exists a unique second-stage equilibrium. In more detail, if $T_1 > T_2$, exactly one of the following two mutually exclusive arrangements of the $\alpha$-intervals will obtain: either (a) $0 < 3T_2 < 2T_1+T_2 < T_1+2T_2+w < 3T_1+w < \infty$ with case $\Gamma_{31}$ empty or (b) $0 < 3T_2 < 3T_2+2w < 3T_1-w < 3T_1+w < \infty$ with case $\Gamma_{22}$ empty. If $T_1 = T_2$, then (a) obtains and the $\alpha$-intervals simplify to $0 < 3T < 3T+w < \infty$. Also, the retailers' equilibrium strategies are continuous at the cutting points of the $\alpha$-intervals, i.e., the ``{\em Range of} $\alpha$'' intervals of \Cref{tab:1} can be taken as left-hand side closed also. Generically, in the $T_1-T_2$ plane, the result of Proposition \ref{prop3.3full} is summarized in \Cref{fig:3}.
\begin{figure}[!htp] \centering
\begin{tikzpicture}[domain=0:12,scale=0.8,thick, yscale=0.9]
\draw[->] (-0.5,0) -- (11,0) node[below] {$T_1$}; 
\draw[->] (0,-0.5) -- (0,9) node[left] {$T_2$};
\draw[dashed] (6,6)--(6,-0.1) node[below]{$\dfrac{\alpha}{3}$}
(6,6)--(5.5,7) node[above] {\footnotesize $\alpha=2T_1+T_2$} 
(4,4)--(3,4.5) node[above] {\footnotesize $\alpha=T_1+2T_2+w\quad$} (8,2)--(9,0);
\draw[thick, domain=0:8.7] plot(\x,\x) node[right]{$T_1=T_2$} 
(-0.1,2)--(0.1,2) node[left]{$\dfrac{\alpha}{3}-\dfrac{2w}{3}\,\,$} 
(-0.1,4)--(0.1,4) node[left]{$\dfrac{\alpha}{3}-\dfrac{w}{3}\,\,$} (-0.1,6)--(0.1,6) node[left]{$\dfrac{\alpha}{3}\,\,$}
(4,4)--(4,-0.1) node[below]{$\dfrac{\alpha}3-\dfrac{w}{3}$} (8,2)--(8,-0.1) node[below]{$\dfrac{\alpha}{3}+\dfrac{w}{3}$}
(8,2)--(11,2) (4,4)--(8,2)--(6,6)--(11,6);
\draw (2.2,1) node{\textbf{$\left(\Gamma_{11}\right)$}} (6,1) node[fill=white]{\textbf{$\left(\Gamma_{21}\right)$}}
(9.5,1) node[fill=white]{$\left(\Gamma_{31}\right)$} (9.5,4) node{$\left(\Gamma_{32}\right)$}
(6,4) node[fill=white]{$\left(\Gamma_{22}\right)$} (9.5,7) node[fill=white]{$\left(\Gamma_{33}\right)$} 
(-0.3, -0.3) node{$0$};
\end{tikzpicture}
\caption{Second stage equilibria regions in the $T_1-T_2$ plane.}
\label{fig:3} \end{figure}
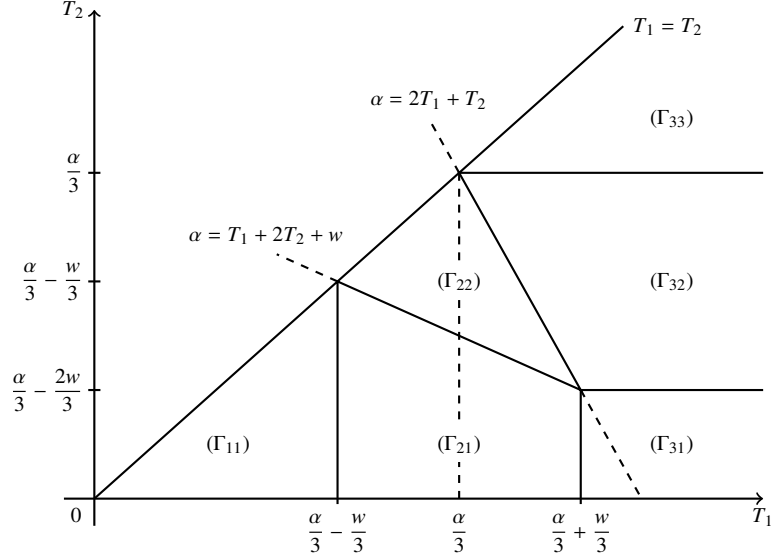
The boundaries of the different regions in \Cref{fig:3} (or equivalently the conditions in \Cref{tab:1}) depend on the values of both $w$ and $\alpha$. However, the point $\frac{\alpha}{3}$ on both the $T_1$ and the $T_2$ axes and the line $\alpha=2T_1+T_2$ that separates the cases $\Gamma_{22}$ and $\Gamma_{32}$ depend only on the value of $\alpha$ and not on $w$. Thus, they will serve as a basis for case discrimination when solving for the optimal strategy of the supplier in the first stage.\par

\subsection{First-stage supplier's equilibrium pricing strategy}\label{sub3.2}
Based on \Cref{fig:3} (or equivalently on \Cref{tab:1}) and using equation \eqref{5}, we can calculate the payoff function, $u^s\(r\mid \alpha\)$, of the supplier as $\alpha$ varies, when the retailers use their equilibrium strategies at the second stage. To proceed, we denote with $q^*_{ij}\(w\):=q^*_1\(w\)+q^*_2\(w\)$ the total quantity that the retailers order from the supplier when case $\Gamma_{ij}$ for $j\le i \in \{1,2,3\}$ occurs in the second stage. Then, we have the following
\begin{lemma}\label{supay}
Assuming that the retailers use their equilibrium strategies in the second stage, the supplier's payoff function is given by
\[u^s\(r\mid \alpha\) = r\cdot \begin{cases} 
q^*_{31}\(r+c\), & 0\le r \le \min{\left\{\frac12 \(\alpha-2c-3T_2\), 3T_1-c-\alpha\right\}}\\
q^*_{21}\(r+c\), & \max{\left\{0, |\alpha-3T_1|-c\right\}}\le r \le \alpha-c-T_1-2T_2\\
q^*_{11}\(r+c\), & 0\le r < \alpha-c-3T_1 \\
0, & \text{else}\end{cases}\] 
\end{lemma}

It is immediate that some cases may not occur for different values of the parameters. Hence, to derive the optimal strategy $r^*$ of the supplier (see proof of Proposition \ref{prop3.4full}), a further case discrimination is necessary. As we have already noted, the relative position of $\alpha$ to the quantities $3T_1, 3T_2$ and $2T_1+T_2$ will serve as a basis for case discrimination. Maximizing the payoff function of the supplier for each case yields his optimal strategy $r^*$ (i.e., the profit margin that maximizes his profits given that the retailers will play their equilibrium strategies in the next stage) for all possible values of $\alpha$ and $T_1, T_2$. \par
To simplify the expressions in the statement of \Cref{prop3.4full}, let $D:=T_1-T_2, \; \Delta:=\frac{\sqrt{3}+3}{2}D$. It will be convenient to distinguish two cases, depending on whether $0<c\le D$ or $0\le D<c$. Also, as mentioned, the second stage equilibrium quantities are continuous at the points at which their expression changes, and therefore, we allow the subsequent cases to overlap on the cutting points. 
\begin{proposition}\label{prop3.4full}
For given $T_1\ge T_2$, the supplier's optimal pricing strategy $r^*$ for all possible values of $\alpha$ is 
\[\begin{array}{ll|ll|ll}
\midrule
&&\textbf{Case:}& r^*\(\alpha\) && \textbf{Conditions}\\ \midrule
\textbf{Case A:} && \text{--}  & r^* \in \mathbb R_+ && 0\le \alpha \le 3T_2+2c\\
0<c\le D &&  \Gamma_{31}: & \frac14\(\alpha-2c-3T_2\) && 3T_2+2c \le \alpha \le 3T_2+\Delta-\frac{\sqrt{3}-1}{2}c\\
&& \Gamma_{21}: & \frac12\(\alpha-c-T_1 -2T_2\) && 3T_2+\Delta-\frac{\sqrt{3}-1}{2}c \le \alpha \le 3T_2+2\Delta+c\\
&& \Gamma_{11}: & \frac12\(\alpha-c-\frac32T_1-\frac32T_2\) && 3T_2+2\Delta+c \le \alpha\\\midrule
\textbf{Case B:} && \text{--} & r^* \in \mathbb R_+ && 0\le \alpha \le T_1+2T_2+c\\
0\le D < c && \Gamma_{21}: & \frac12\(\alpha-c-T_1 -2T_2\) && T_1+2T_2+c \le \alpha \le 3T_2+2\Delta+c\\
&& \Gamma_{11}: & \frac12\(\alpha-c-\frac32T_1-\frac32T_2\) && 3T_2+2\Delta+c \le \alpha\\\midrule
\end{array}\]\vspace{-0.3cm}
\captionsetup{type=table} \captionof{table}{First stage equilibrium strategies for all $\alpha\ge0$.}
\end{proposition}
Combining \Cref{prop3.3full,prop3.4full}, one obtains the subgame perfect equilibria of the two-stage game under complete information for the general case of asymmetric retailers and for all different values of $\alpha,c$ and $T_1\ge T_2$.
\subsection{Higher demand does not imply a higher wholesale price}
Although the supplier's payoff function $u^s\(r\mid \alpha\)$ is continuous at the cutting points of the $\alpha$-intervals, the same is {\em not} true for the supplier's optimal strategy as can be checked from Proposition \ref{prop3.4full}. For given $T_1, T_2$ with $T_1 \ge T_2$, not only $r^*\(\alpha\)$ is not continuous in $\alpha$, but it is also not increasing in $\alpha$ (although it is increasing on each sub-interval). At the points of discontinuity the supplier is indifferent between the left-hand side and right-hand side strategies. However, any mixture of these strategies does not yield the same payoff since his payoff is not linear in $r^*$. The reason for these discontinuities is that the supplier faces a piecewise linear demand instead of a linear demand. Therefore, at the cutting points there is a ``jump'' from the $\argmax$ of the first linear part to the $\argmax$ of the other linear part, hence the discontinuity. This also explains the decrease of the optimal price as we move from certain cutting points to the right, even if the increase in the demand intercept $\alpha$ is $\epsilon$-small. Of course, one can check that under his optimal strategy, not only the supplier's payoff is continuous, but it is also an increasing function of the demand $\alpha$, as expected. Formally,
\begin{corollary} The retailers' equilibrium strategies and payoff functions are continuous in the demand level $\alpha$. While the supplier's payoff is also continuous and increasing in $\alpha$, his equilibrium pricing policy is neither continuous nor monotonic in $\alpha$.
\end{corollary}

\subsection{The symmetric case: identical second-stage retailers}\label{symmetric}
The equilibrium analysis considerably simplifies if we restrict attention to the case of identical retailers, i.e., $T_1=T_2=T$. In this case, only symmetric equilibria may occur in the second stage. The equilibrium strategies depend on the value of $\alpha$ and its relative position to $3T$. The proofs of \Cref{prop3.3,prop3.4} follow immediately from \Cref{prop3.3full,prop3.4full} and are omitted.
\begin{corollary}\label{prop3.3}
If $T_1=T_2=T$, then for all values of $\alpha$, the second stage equilibrium strategies between retailers $R_1$ and $R_2$ are symmetric, and for $i=1,2$ they are given by $s^*_i\(w\)=\(t^*_i\(w\), q^*_i\(w\)\)$ with $t^*_i\(w\)= T-\frac13\(3T-\alpha\)^+$, and $q^*_i\(w\)= \frac13\(\alpha-3T-w\)^+$. 
\end{corollary}
The general form of the supplier's payoff function $u^s\(r\mid \alpha\)$ is given by \eqref{5} and his strategy set by \eqref{4}. Obviously, the supplier will not be willing to charge prices lower than his cost $c$. Based on the discussion of \Cref{prop3.3} and the constraint $w\ge c$ (i.e., $r\ge 0$), we conclude that a transaction will take place for values of $\alpha>3T+c$ and for $r\in [0, \alpha -3T -c)$. In that case, the optimal profit margin of the supplier is at the midpoint of the $r$-interval.\par
To see this, let $q^*\(w\):=q^*_1\(w\)+q^*_2\(w\)$ denote the total quantity that the supplier will receive as an order from the retailers when they respond optimally. By \Cref{prop3.3}, $q^*(w)=\frac23\(\alpha-3T-w\)^+$. Hence, on the equilibrium path and for $r \ge 0$, the payoff of the supplier is
\begin{equation}\label{pay}u^s\(r\mid \alpha\)=rq^*\(w\)=\frac23r\(\alpha-3T-c-r\)^+\end{equation}
We then have
\begin{corollary}\label{prop3.4}
For $T_1=T_2=T$ and for all values of $\alpha$, the subgame perfect equilibrium strategy $r^*\(\alpha\)$ of the supplier is given by $r^*\(\alpha\)=\frac12\(\alpha-3T-c\)^+$.
\end{corollary}
\Cref{prop3.4} implies that if $\alpha<3T+c$, then the optimal profit of the supplier is equal to $0$, i.e., he will set a price equal to his cost. Actually, he is indifferent between any price $w\ge c$ since in that case, he knows that the retailers will order no additional quantity. In sum, \Cref{prop3.3,prop3.4} provide the subgame perfect equilibrium of the two-stage game in the case of identical (i.e., $T_1=T_2=T$) retailers. 
\begin{theorem}\label{thm3.5} If the capacities of the retailers are identical, i.e., if $T_1=T_2=T$, then the complete information two-stage game has a unique subgame perfect Nash equilibrium, under which the supplier sells with profit margin $r^*\(\alpha\)=\frac12\(\alpha-3T-c\)^+$ and each of the retailers orders quantity $q^*\(w\)= \frac13\(\alpha-3T-w\)^+$ and produces (releases from his inventory) quantity $t^*\(w\)= T-\frac13\(3T-\alpha\)^+$.
\end{theorem}
\subsection{Mitigating marginal cost effects: increased social welfare}
Using the notation of \Cref{prop3.3}, $Q^*_i=t^*_i+q^*_i$ is the total quantity of the good that each retailer releases to the market in equilibrium. If $\alpha \le 3T$, then $Q^*_i$ is equal to the equilibrium quantity of a classic Cournot duopolist who faces linear inverse demand with intercept equal to $\alpha$ and has $0$ cost per unit. If $\alpha>3T$, then the equilibrium quantities of the retailers depend on the supplier's price $w$. If $w \ge\alpha-3T$, the retailers will avoid ordering and will release their inventories to the market. Contrarily, if $w$ is low enough, i.e., if $w<\alpha-3T$, they will be willing to order additional quantities from the supplier. In this case, $Q^*_i=t^*_i+q^*_i=T+\frac13\(\alpha-3T-w\)=\frac13\(\alpha-w\)$. This quantity is equal to the equilibrium quantity of a Cournot duopolist who faces linear demand with intercept $\alpha$ and cost per unit $w$ for all product units, despite the fact that here, the retailers face a cost of $0$ for the first $T$ units and $w$ for the rest. Hence, in this case and as a result of marginal cost analysis, the market behaves \emph{as if} there were no inventories, and the most expensive units are the ones that determine the total quantity sold to the market. However, if the market coordinates on the unique subgame perfect equilibrium, then as determined in \Cref{thm3.5}, the supplier's optimal price, $w^*\(\alpha\)=r^*\(\alpha\)+c=\frac12\(\alpha-3T-c\)^++c$, depends negatively on $T$, i.e., it decreases as $T$ increases.\par
Based on these observations, the retailers' low-cost in-house production (or low-cost inventory) affects in different ways the total quantity that is released to the market and thus, it has both a direct and an indirect impact on the consumers' surplus\footnote{We remind that the consumers' surplus is proportional to the square of the total quantity that is released to the market.}: for lower values of the demand parameter, the consumers benefit from the retailers' low-cost inventories or production capacities since otherwise no goods would have been released to the market. For higher values of $\alpha$ the total quantity that is released to the market depends indirectly on $T$ -- via a lower wholesale price set by the supplier -- and thus the benefits from the retailers' integrated low-cost production channel are again experienced by the consumers.

\section{Subgame-perfect equilibrium under demand uncertainty}\label{sec4}
We now study the equilibrium behavior of the supply chain, assuming that the supplier has incomplete information about the true value of the demand parameter $\alpha$ when he sets his price, while the retailers know it when they place their orders (if any). In the two-stage game context, we assume that the demand is realized after the first stage (i.e., after the supplier sets his price) but prior to the second stage of the game (i.e., prior to the decision of the retailers about the quantity they will release to the market). The case $T_1\ge T_2$ exhibits significant computational difficulties and we will restrict our attention to the symmetric case $T_1=T_2=T$. \par
Under these assumptions, the equilibrium analysis of the second stage (as presented in \cref{sub3.1}) remains unaffected: the retailers observe the actual demand parameter $\alpha$ and the price $w$ set by the supplier and choose the quantities $t_i$ and $q_i$ for $i=1,2$. However, in the first stage, the actual payoff of the supplier depends on the unknown parameter $\alpha$ for which he has a belief: the distribution $F$ of $\alpha$, which induces a non-atomic measure on $[0, \infty)$ with finite expectation, as assumed. Given the value of $\alpha$ and assuming that the retailers respond to the supplier's choice $w$ with their unique equilibrium strategies, the supplier's actual payoff is provided by equation \eqref{pay}. Taking the expectation with respect to his belief (see also \eqref{6}) implies that when $\alpha$ is unknown and distributed according to $F$, the payoff function of the supplier is equal to 
\begin{equation}\label{8} \us=\frac23r \,\ex\(\alpha-3T-c-r\)^{+} \;\; \mbox{for $r \ge 0$}. \end{equation}
Let $r_H:=\alpha_H-3T-c$ and $r_L:=\alpha_L-3T-c$. If $r_H \leq 0$, then $\alpha \leq 3T+c $ for all $\alpha$, i.e., then, $\us\equiv 0$ and the problem is trivial. Hence, to proceed, we assume that $r_H>0 \Leftrightarrow 3T+c<\alpha_H$. Then, since $\us =0$ for $r\ge r_H$, we may restrict the domain of $\us$ and take it to be the interval $[0,r_H)$ in \eqref{8}. To proceed, we observe that $\us$ may be expressed in terms of the Mean Residual Lifetime (MRL) function, which is defined as, see e.g., \cite{Sh07} or \cite{Be16},
\begin{equation}\label{imrl} \m\(t\):=\begin{cases}\,\ex\(\alpha-t \mid \alpha >t\)=\dfrac{\int_{t}^{\infty}\(1-F\(x\)\)dx}{1-F\(t\)} & \mbox{if $P\(\alpha >t\)>0$} \\
\,0 & \mbox{otherwise} \end{cases}
\end{equation}
The term MRL stems from the reliability and actuarial literature. In the present setting, the MRL function denotes the expected additional demand. Despite being a powerful and well-studied tool to deal with uncertainty, the MRL function has only scarce and informal applications in the revenue management literature or economics in general. Using this representation, one may get the following\footnote{It should be stressed that obtaining \eqref{der} is not straightforward at all, because the product rule of differentiation does not apply. In particular, see \Cref{lem diffa} and the proof of \Cref{lem diff}, all in the Appendix.}
\begin{lemma}\label{lem diff}
For $r \in \(0,r_H\)$, the supplier's payoff function and its derivative are expressed through the MRL function by $\us=\frac23r\m\(r+3T+c\)\(1-F\(r+3T+c\)\)$ and \begin{equation}\label{der}\fdr=\frac23\(\m\(r+3T+c\)-r\)\(1-F\(r+3T+c\)\),\end{equation} respectively. In addition, the roots of $\fdr$ in $\(0,r_H\)$ (if any) satisfy the fixed point equation $r^*=\m\(r^*+3T+c\)$.
\end{lemma}
Based on the first order conditions stated in \Cref{lem diff}, we are now able to derive necessary and sufficient conditions for the existence and uniqueness of a Bayes-Nash equilibrium. 
\begin{theorem}[Necessary and sufficient conditions for Bayesian Nash equilibria]\label{thm4.1} 
Under incomplete information with identical retailers (i.e., for $T_1=T_2=T$) for the non-trivial case $r_H > 0$ and assuming the supplier's belief induces a non-atomic measure on the demand parameter space:
\begin{alphlist}[(a)]
\item{\bf (necessary condition)} If the optimal profit margin $r^*$ of the supplier exists when the retailers follow their equilibrium strategies in the second stage, then it satisfies the fixed point equation
\begin{equation}\label{14}r^*=\m\(r^*+3T+c\)\end{equation}
\item{\bf (sufficient condition)} If the mean residual lifetime $\m\(\cdot\)$ of the demand parameter $\alpha$ is decreasing, then the optimal profit margin $r^*$ of the supplier exists under equilibrium, and it is the unique solution of the equation $r^*=\m\(r^*+3T+c\)$. In that case, if $\ex\(\alpha\) - \alpha_L \leq \alpha_L -3T-c \ \(=r_L\)$, then $r^*$ is given explicitly by $r^*=\frac12\(\ex\(\alpha\)-3T-c\)$. Otherwise $r^* \in \(r_L^+, r_H\)$.
\end{alphlist}
\end{theorem}
\Cref{prop3.3} and \Cref{thm4.1} lead to
\begin{corollary}\label{cor4.2} If the capacities of the producers (retailers) are identical, i.e., if $T_1=T_2=T$, and if the distribution $F$ of the demand intercept $\alpha$ is of decreasing mean residual lifetime $\m\(\cdot\)$, then the incomplete information two-stage game has a unique subgame perfect Bayesian Nash equilibrium for the non-trivial case $r_H > 0$. At equilibrium, the supplier sells with profit margin $r^*$, which is the unique solution of the fixed point equation $r^*=\m\(r^*+3T+c\)$, and each of the producers (retailers) orders quantity $q^*\(w\)= \frac13\(\alpha-3T-w\)^+$ and produces (releases from his inventory) quantity $t^*\(w\)= T-\frac13\(3T-\alpha\)^+$.
\end{corollary}
\noindent By the DMRL property, \Cref{thm4.1} implies that $r^*\le \ex\(\alpha\)$, since
\begin{equation}\label{16}r^*=\m\(r^*+3T+c\)\le \m\(u\)\le \m\(0\)=\ex\(\alpha\),\end{equation}
for all $u$ such that $0\le u \le r^*+3T+c$. Finally, if $T=0$, then our model corresponds to a classic Cournot duopoly at which the retailers' cost equals the price that is set by a single revenue-maximizing supplier. As shown in our related study, \cite{Le17}, in this case, the milder condition of \emph{decreasing generalized mean residual life (GMRL)} is sufficient to yield existence and uniqueness of Bayes-Nash equilibrium.

\section{Why the supplier may charge a high price}\label{secinef}
Utilizing the characterization of the supplier's optimal price in \Cref{thm4.1}, we compare the complete and incomplete information equilibrium outcomes. It is well known that markets with incomplete information may be inefficient at equilibrium in that trades that would be beneficial for all players may not occur under Bayesian Nash equilibrium, e.g., see \cite{My83}. Namely, under equilibrium, there exist values of $\alpha$ for which a transaction would have occurred in the complete information case but not in the incomplete information case. In such a case, the market will experience a \emph{stockout} due to lack of coordination between the supplier and the retailers. We call this phenomenon, an \emph{incoordination stockout}. 

\subsection{Measuring market inefficiency}
For the incomplete information case and for a particular distribution $F$ of $\alpha$, let $U$ be the event that a transaction would have occurred under equilibrium if we had been in the complete information case, and let $V$ be the event that a transaction does not occur in the incomplete information case under equilibrium. Our aim is to measure the inefficiency of our market by studying $P\(V \cap U\)$ and $P\(V \mid U\)$. By \Cref{prop3.3}, a transaction will take place under equilibrium if and only if
\begin{equation}\label{ntr} \alpha > r^* +3T +c\end{equation}
where $r^*$ stands for the profit margin of the supplier in the incomplete information case under equilibrium. Using (\ref{ntr}) and \Cref{prop3.4}, a necessary and sufficient condition for a transaction to take place under equilibrium in the complete information case is $\alpha > 3T +c$.  Using \Cref{thm4.1} and the steps at the end of its proof, it is straightforward to check that if $\ex\(\alpha\) - \alpha_L \leq r_L$, then \eqref{ntr} is satisfied for all $\alpha > \alpha_L$  since the ordering is $3T +c < r^* +3T +c \leq \alpha_L$. Hence $V \cap U=\emptyset$. So let us assume that $\ex\(\alpha\) - \alpha_L > r_L$. We then have

\begin{lemma}\label{lemprob}
For a given distribution $F$ of $\alpha$ with the DMRL property, let $V$ be the event that a transaction does not occur in the incomplete information case under equilibrium and let $U$ be the event that a transaction would have occurred under equilibrium if we had been in the complete information case. Then
\[P\(V \mid U\) = \frac{F\(r^* +3T +c\) - F\(3T +c\)}{1-F\(3T +c\)}\]
In particular, if \, $ \ex\(\alpha\) - \alpha_L \leq \alpha_L -3T-c \ \(=r_L\)$, then $P\(V \cap U\) = P\(V \mid U\) = 0$. 
\end{lemma}
As in the case when $T=0$ (see \cite{Le17}) the next Theorem shows that $P\(V\mid U\)$ admits a bound which is independent of $F$.

\begin{theorem}\label{thmbound}
For any distribution $F$ of $\alpha$ with the DMRL property, the probabilitiy of an incoordination stockout, i.e. the conditional probability $P\(V \mid U\)$ that a transaction does not occur under equilibrium in the incomplete information case, given that a transaction would have occurred under equilibrium if we had been in the complete information case, cannot exceed the bound $1-e^{-1}$, i.e.,
\begin{equation}\label{bound}P\(V\mid U\)\le 1-e^{-1}\end{equation}
This bound is tight over all DMRL distributions, because it is attained by the exponential distribution.
\end{theorem}

\subsection{Risk of an incoordination stockout}\label{risk}
To gain more intuition about the reasons that make the supplier charge a price that runs a (sometimes) considerable risk of no transaction, although such a transaction would be beneficial for all market participants, recall the discussion preceding Lemma \ref{lemprob}. To have $P\(V \mid U\) = 0$, the expectation restriction, $\ex(a) \leq \alpha_L + r_L$, must apply. So, assuming this restriction applies and keeping everything else the same (i.e., $\alpha_L$, $T$ and $c$), start moving probability mass to the right. Then, $\ex \(\alpha\)$ will be increasing which results in the supplier's optimal profit margin, $r^* = \frac{1}{2} \(\ex\(\alpha\) -3T -c\)$, increasing. Thus, there is a threshold for $\ex\(\alpha\)$, namely $\alpha_L + r_L$ (when the supplier charges $r^* = r_L$), above which the supplier is willing to charge a price so high (i.e., above $r_L$) that he will run the risk of receiving no orders. Notice, that this discussion implies neither monotonicity of the supplier's payoff nor monotonicity of the probability of no transaction. \par
Next, assume that the expectation restriction applies and keeping $T$ and $c$ the same, start moving $\alpha_L$ to the left. Now, $\ex \(\alpha\)$ will be decreasing, which results in the supplier's optimal profit margin, $r^*$, decreasing. However, since $r_L$ will eventually become non-positive while $\ex\(\alpha\) - \alpha_L$ always stays positive, it is easy to see that again, the same threshold applies for the expectation condition, i.e., eventually we will get $\ex\(\alpha\) = \alpha_L + r_L$, the supplier will charge $r^* = r_L$ there, and for lower values of $\alpha_L$ inefficiencies will appear. In other words, for values of $\alpha_L$ below the threshold, although he is reducing his profit margin, the supplier is asking for a {\em relatively} high price (i.e., above $r_L$) and is willing to take the risk of no transactions. Finally, using similar arguments, we can see that if we either increase the inventory level $T$ of the retailers or the cost $c$ of the supplier, all else being kept the same (see also \Cref{inventory,cost}), the same threshold applies for the appearance of inefficiencies, determined by $\ex\(\alpha\)=\alpha_L+r_L$.

\subsection{Numerical simulations}
In the simulations presented below, we examine the probability of no transaction, as determined in \Cref{lemprob} and the supplier's expected payoff against the expected demand (mean of demand parameter $\alpha$). First, we simulate the uniform distribution $\alpha\sim U[\alpha_L,\alpha_H]$, i.e., $f\(\alpha\)=\frac{1}{\alpha_H-\alpha_L}\cdot\mathbf{1}_{\alpha \in [\alpha_L,\alpha_H]}$ with $\alpha_L=100$ and $\alpha_H=100+5\cdot j$ for $j=1,\dots,100$, cost $c=1$ and inventories $T=20$. The results are shown in \Cref{uniprb} below. The $x-$axis represents the mean $\ex \alpha$ of the respective distribution of $\alpha$. Observe that the probability of no transaction becomes positive for $\ex \alpha >139 \(=\alpha_L+r_L\)$. The supplier's expected payoff is strictly positive for all $j\ge 1$ and moreover, it increases despite an increasing probability of no transaction. 
\begin{figure}[!htp]
\centering
\includegraphics[width=0.82\linewidth]{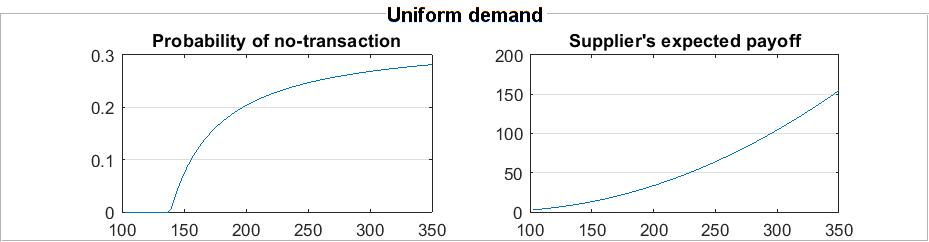}
\caption{No transaction probability and expected supplier's payoff as the expected demand increases.}
\label{uniprb}
\end{figure}
The same behavior is exhibited by the (shifted) exponential distribution. In \Cref{expprb}, we see the simulation results for $\alpha \sim 100+\exp\(\lambda\)$, i.e., $f\(\alpha\)=\frac1{\lambda} e^{-\frac{1}{\lambda} \(x-100\)}$ for $x\ge 100$, with $\lambda=10\cdot j$ for $j=1,\dots 50$. As above, $T=20$ and $c=1$ which implies that the threshold for the appearance of inefficiencies, remains the same, i.e., $\ex \alpha>139 \(=\alpha_L+r_L\)$. In the left figure, we observe that the probability of no transaction attains the upper bound that is determined in \Cref{thmbound}. Having a constant MRL, the shifted-exponential distribution is DMRL, which shows that this bound is tight over the class of DMRL distributions. Again, the supplier's expected payoff -- right figure -- is strictly positive for all $j\ge1$ and increasing. 
\begin{figure}[!htp]
\centering
\includegraphics[width=0.82\linewidth]{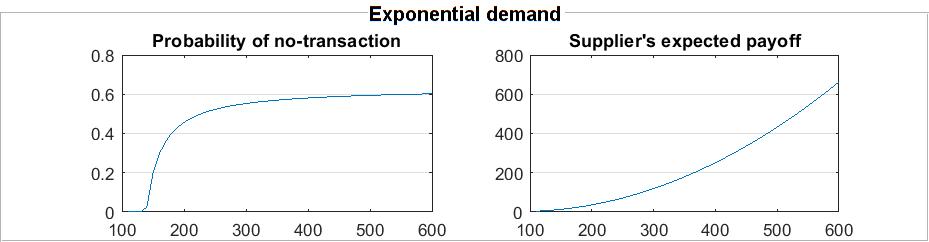}
\caption{Exponential distribution: $\alpha\sim 100+\exp\(\lambda\)$.}
\label{expprb}
\end{figure}

\subsection{Lower expected demand but higher wholesale price}
The next simulation, shows that a higher expected demand does not necessarily imply a higher wholesale price. To see this, we compare two gamma distributions $\alpha_1$ and $\alpha_2$ for the market demand, with parameters $\alpha_1\sim 100+\Gamma\(10,2\)$ and $\alpha_2 \sim 100+\Gamma\(2, 9\)$. We say that $\alpha \sim \Gamma\(k,\theta\)$ if $\alpha$ has probability density function $f\(r\mid k,\theta\)=\frac{1}{\Gamma\(k\)\theta^k}r^{k-1}e^{-\frac{r}{\theta}}$. Here, $T=33$ and $c=1$.
\begin{figure}[!htp]
\centering
\includegraphics[width=0.8\linewidth]{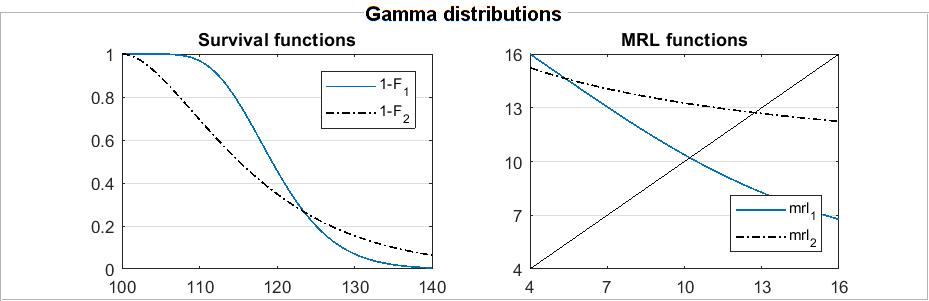}
\caption{Gamma distributions: $\alpha_1\sim \Gamma\(10,2\)$ and $\alpha_2\sim \Gamma\(2,9\)$.}
\label{gammas}
\end{figure}
The figure on the left, shows that $\alpha_1$ and $\alpha_1$ are incomparable in the usual stochastic order. The figure on the right shows the MRL functions $\m_1\(r\)$ of $\alpha_1$, $\m_2\(r\)$ of $\alpha_2$ and the diagonal $y=r$. The intersection points are precisely the optimal wholesale prices for each distribution. Here $r^*_1=10.21$ and $r^*_2=12.73$, despite the fact that $\ex \alpha_1=20>18=\ex \alpha_2$. The supplier's expected payoff is higher in the first market than in the second, since $u^s(r^*_1)=100.5>95.1=u^s(r^*_2)$. \par
The observation that larger markets (in terms of expectation) do not necessarily give rise to higher prices confirms the intuition of \cite{La01} that "size is not everything" and that price movements are driven by different forces. In our case the explanation is provided by \eqref{14}: to obtain a higher price one needs to compare two markets in terms of the mrl-order (see \cite{Sh07}) and not in terms of their means or of the usual stochastic order. Here, $\alpha_2$ \emph{eventually} dominates $\alpha_1$ in terms of their mrl functions and hence the market that is described by $\alpha_2$ results in a higher wholesale price. 


\section{Effect of cost and inventory size on the supplier's profit margin and pricing policy}\label{seccost}
The characterization result of \Cref{thm4.1} facilitates the comparative statics analysis on the parameters that determine the supplier's pricing policy. Under the DMRL assumption, if everything else stays the same but retailers' inventories are rising (but staying  below $(\alpha_H -c)/3$ because of the non-triviality assumption), the supplier will strictly decrease or keep constant the price he charges at equilibrium, because his profit margin $r^*$ will be non-increasing. The reason is that as $T$ increases, the graph of $\m\(\,\cdot \, , T, c\):=\m\(\,\cdot\,+3T+c\)$ shifts to the left, and therefore, its intercept with the line bisecting the first quadrant decreases\footnote{To avoid confusion, we use the terms ``decreases'' in the sense ``decreases non-strictly'' or ``does not increase'', as we did before.}. By \Cref{thm4.1}, this intercept is $r^*$, and hence the price $w^*=c+r^*$ the supplier asks at equilibrium will be decreasing. Hence,
\begin{corollary}\label{inventory}
Under the DMRL property, if everything else stays the same and the retailers' inventory capacity $T$ increases in the interval $\left[0, \(\alpha_H -c\)/3\)$, then at equilibrium, the supplier's profit margin $r^*$, and hence the wholesale price $w^*=r^*+c$, both decrease.
\end{corollary}
The result of \Cref{inventory} is illustrated in \Cref{expnv} for exponentially distributed demand with mean $\lambda=10$. The supplier's cost is constant and equal to $c=1$, and the retailers' inventory is equal to $T=j$, with $j$ taking values in $j=1,2,\dots, 30$. 
\begin{figure}[!ht]
\centering
\includegraphics[width=0.8\linewidth]{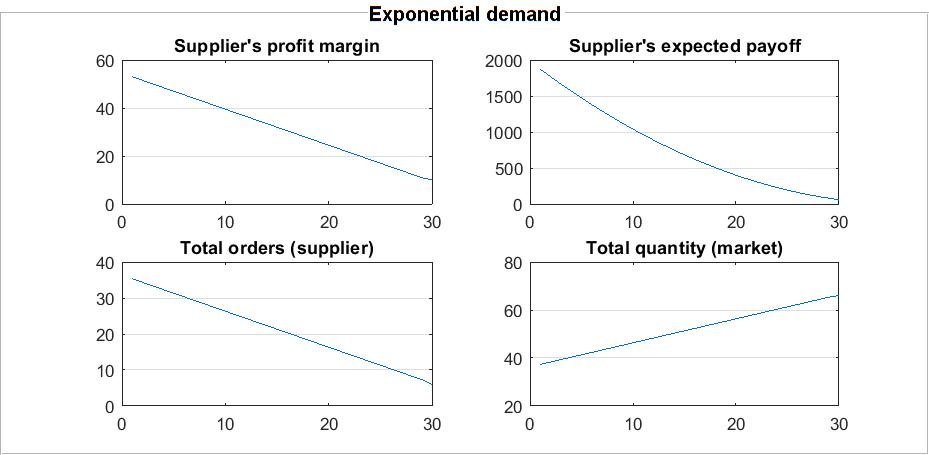}
\caption{Basic model quantities as retailer's inventory increases.}
\label{expnv}
\end{figure}
\par
By the same argument, if we take the supplier's cost $c$ to be increasing (but staying below $\alpha_H -3T$) while everything else stays fixed, then the supplier's profit margin $r^*$ will again be decreasing. However, this time, the price $w^*=c+r^*$ the supplier asks at equilibrium will be increasing. To see this, let $c_1 < c_2$. Then, since $\m\(\cdot\)$ is decreasing, $r_2^* \leq r_1^*$. If $r_2^* = r_1^*$, then $w_1^* = r_1^* + c_1 < r_2^* + c_2 = w_2^*$. If $r_2^* < r_1^*$, by \Cref{thm4.1}, $\m(r_2^* + 3T + c_2) < \m(r_1^* + 3T + c_1)$. The DMRL property then implies that $r_1^* + c_1 \leq r_2^* + c_2$, i.e., $w_2^* \geq w_1^*$. So, in this case, we have
\begin{corollary}\label{cost}
Under the DMRL property, if everything else stays the same and the supplier's cost $c$ increases in the interval $[0, \alpha_H - 3T)$, then at equilibrium, the supplier's profit margin $r^*$ decreases while the wholesale price $w^*=r^*+c$ increases.
\end{corollary}
The result of \Cref{cost} is presented graphically below. Let $\alpha\sim 100+100\cdot Beta\(1,2\)$, i.e., the demand parameter $\alpha$ follows a scaled Beta distribution on the interval $\left[\alpha_L,\alpha_H\right]=\left[100,200\right]$ and let $T=40$. The $x-$axis corresponds to values for the cost parameter $c$ in the range $c=1,2,\dots,30$. 
\begin{figure}[!ht]
\centering
\includegraphics[width=0.8\linewidth]{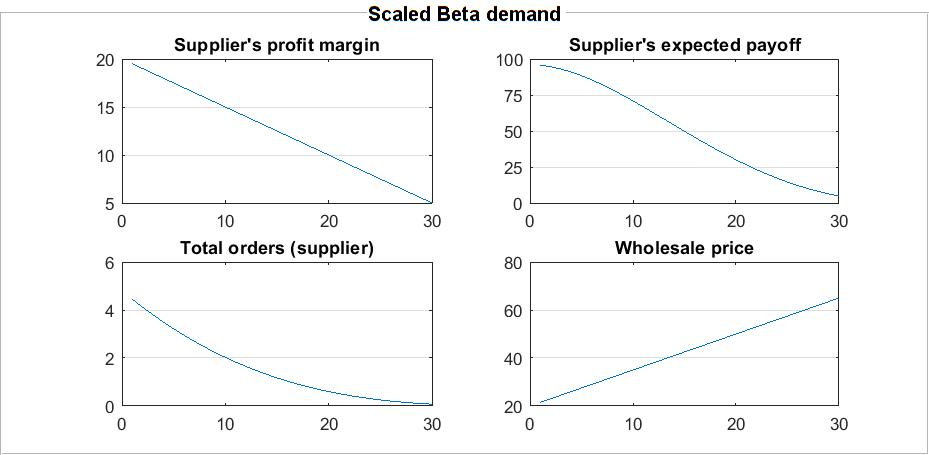}
\caption{Basic model quantities as supplier's cost increases (scaled Beta demand)}
\label{scaledbeta}
\end{figure}
As predicted by \Cref{cost}, the wholesale price that the supplier charges increases, whereas his profit margin decreases. The total order quantity (quantity ordered by both retailers) decreases and since the inventory parameter $T$ is kept constant, this results in a decrease in consumers' surplus. A similar behavior of these quantities is observed when $\alpha\sim 100+\exp\(\lambda=10\)$, i.e., when $\alpha$ follows a shifted exponential distribution with mean $\lambda=10$ and $T=40$. For completeness, the results of the simulation are shown in \Cref{expcst} below
\begin{figure}[!ht]
\centering
\includegraphics[width=0.8\linewidth]{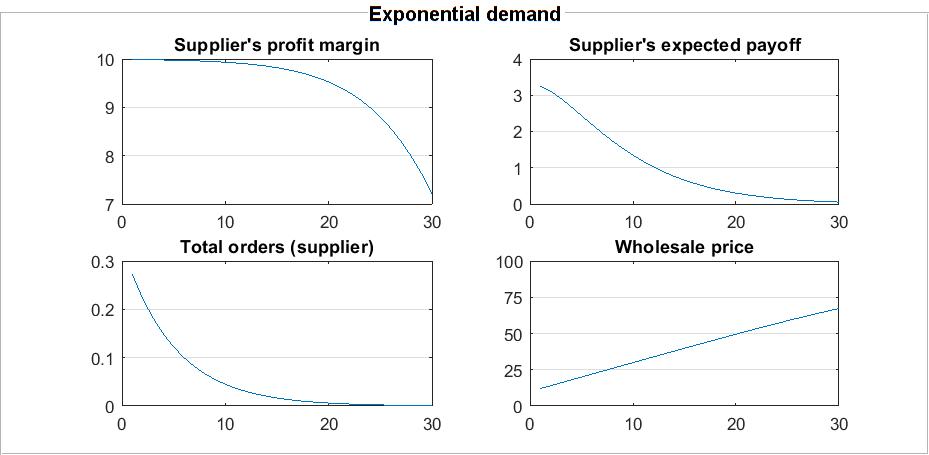}
\caption{Basic model quantities as supplier's cost increases (exponential demand)}
\label{expcst}
\end{figure}
Further simulations with uniformly and Pareto distributed demand -- which does not satisfy the DMRL property -- and for a wide range of their parameters essentially highlight the same qualitative behavior for these quantities (wholesale price, profit margin, supplier's payoff and total order quantities) and thus, are omitted. 

\section{Number of second-stage retailers and supplier's profits}\label{sec7}
\Cref{thm4.1}, the main result of \Cref{sec4}, admits a straightforward extension to the case of $n>2$ identical retailers, i.e., $n$ retailers each having capacity constraint $T_i=T$. Formally, let $N=\{1,2, \ldots, n\}$, with $n\ge 2$ and denote with $R_i$ retailer $i$, for $i\in N$. As in \Cref{secmodel}, a strategy profile is denoted with $s=\(s_1, s_2, \ldots, s_n\)$. The payoff function of $R_i$ depends on the total quantity of the remaining $n-1$ retailers and is given by \eqref{3}, where now $Q$ denotes the total quantity sold by all $n$ retailers, i.e., 
\begin{equation*}Q=\sum_{j=1}^nQ_j=\sum_{j=1}^n (t_j+q_j)
\end{equation*} 
Following common notation, let $s=\(s_{-i},s_i\)$ and $Q_{-i}=Q-Q_i$ for $i \in N$. It is immediately evident that \Cref{lem3.1} and \Cref{lem3.2} still apply, if one replaces $Q_j$ with $Q_{-i}$ and $s_j$ with $s_{-i}$. Hence, one may generalize Proposition \ref{prop3.3} as follows

\begin{proposition}\label{prop7.1}
If $T_i=T$ for $i \in N$, then for all values of $\alpha$ the strategies $s^*_i\(w\)=\(t^*_i\(w\), q^*_i\(w\)\)$, or shortly $s^*_i=\(t^*_i, q^*_i\)$, with $t^*_i(w)= T-\frac{1}{n+1}\(\(n+1\)T-\alpha\)^+$, and $q^*_i\(w\)= \frac{1}{n+1}\(\alpha-\(n+1\)T-w\)^+$, are second-stage equilibrium strategies among the retailers $R_i$, for $i \in N$.
\end{proposition}
Turning attention to the first stage, the payoff function of the supplier in the complete information case (cf. \Cref{sub3.2}) will be given by $\us=rq^*\(w\)=\frac{n}{n+1}r\(\alpha-\(n+1\)T-c-r\)^+$, and hence, it is maximized at $r^*\(\alpha\)=\frac12\(\alpha-\(n+1\)T-c\)^+$, which generalizes Proposition \ref{prop3.4}. Similarly, if the supplier knows only the distribution and not the true value of $\alpha$, the arguments of \Cref{sec4} still apply. Then the payoff function of the supplier - cf. \eqref{8} - becomes $\us=\frac{n}{n+1}r \,\ex\(\alpha-\(n+1\)T-c-r\)^{+}$, for $r \ge 0$ and hence 
\begin{equation*}\us=\frac{n}{n+1}r \,\m\(r+\(n+1\)T+c\)\(1-F\(r+\(n+1\)T+c\)\), \end{equation*}
for $0\le r<\alpha_H-\(n+1\)T-c$. Proceeding as in \Cref{sub3.2}, we can generalize \Cref{thm4.1} to the case of $n\ge 2$ identical retailers. To this end, let $r_L^{\,n}:=\alpha_L-\(n+1\)T-c$ and $r_H^{\,n}:=\alpha_H-\(n+1\)T-c$. Then,
\begin{theorem}[Necessary and sufficient conditions for Bayesian Nash equilibria]\label{thmnretailers}
Under incomplete information, for identical retailers (i.e., for $T_i=T,\,i \in N$) for the non-trivial case $r_H^{\,n} >0$ and assuming the supplier's belief induces a non-atomic measure on the demand parameter space:
\begin{alphlist}[(a)]
\item{\bf (necessary condition)} If the optimal profit margin $r^*$ of the supplier exists when the retailers follow their equilibrium strategies in the second stage, then it satisfies the fixed point equation \[r^*=\m\(r^*+\(n+1\)T+c\)\]
\item {\bf (sufficient condition)} If the mean residual lifetime $\m\(\cdot\)$ of the demand parameter $\alpha$ is decreasing, then the optimal profit margin $r^*$ of the supplier exists under equilibrium, and it is the unique solution of the equation $r^*=\m\(r^*+\(n+1\)T+c\)$. In that case, if $\ex\(\alpha\) - \alpha_L \leq \alpha_L -\(n+1\)T-c\ \(=r_L^{\,n}\)$, then $r^*$ is given explicitly by $r^*=\frac12\(\ex\(\alpha\)-\(n+1\)T-c\)$. Otherwise, $r^* \in \(\(r_L^{\,n}\)^+, r_H^{\,n}\)$.
\end{alphlist}
\end{theorem}
\subsection{Increased competition may benefit the supplier}
Based on \Cref{thmnretailers}, we study how the number of second-stage retailers affects supplier's profit and consumers' surplus, the latter as expressed by the total quantity that is sold in the market. In \Cref{nequil}, we present the results from simulating the distribution of $\alpha$ as the uniform distribution on the interval $\left[\alpha_L,\alpha_H\right]=\left[100,300\right]$, with $c=1$ and $T=5$ ($T$ represents the quantity produced by each of the $n$ identical retailers).
\begin{figure}[!ht]
\centering
\includegraphics[width=0.8\linewidth]{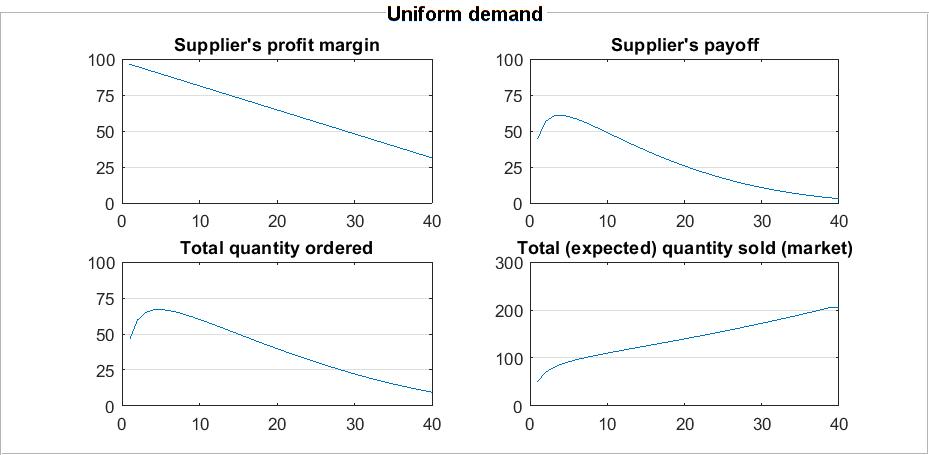}
\caption{Basic model quantities for increasing number of second-stage retailers.}
\label{nequil}
\end{figure}
For values of $n\le 5$, the supplier benefits from increasing second-stage competition both in terms of quantity ordered and in terms of his overall payoff. However, as the number of retailers further increases -- and hence the total quantity produced (or held as inventory) also increases -- the supplier receives less total orders and his payoff declines. On the other hand, the total quantity that is released to the market increases as the number of retailers increases. Under our assumptions, each additional retailer has access to a constant amount of $T$ low-cost in-house produced units. This leads to a lower profit margin for the supplier and a lower overall cost for the retailers that, in turn, leads to the observed increase of the total quantity that is sold to the consumers. Hence, in the present setting, increased second-stage competition positively contributes to the total consumer surplus.\par
The behavior exhibited in \Cref{nequil} seems typical -- based on several simulations that are omitted for brevity -- for a wide range of parameters and distributions that may (Beta, exponential) or may not (Generalized Pareto) satisfy the DMRL condition. 

\section{Discussion \& future work}\label{discussion}

From a game-theoretic perspective, the current model aims to capture and study the complex cost structure of contemporary Cournot oligopolists through a mathematical model that is realistic, but also tractable. The derivation of the market equilibrium under different information levels and the analytical representation of the equilibrium strategies of the interacting entities, supplier and retailers, provide the proper framework for an extensive sensitivity analysis of the model parameters. \par
From an economic perspective, the model assumptions and the setup that is employed aim to reflect situations from the real economic practice and thus, to have an impact on policy development and implementation by active economic stakeholders. To this end, we further elaborate on the model assumptions and discuss directions for further research.

\textbf{Single supplier:} The monopoly setting that we employ in the present paper allows for a transparent study on the pricing decisions of a firm. Similar models, with a single supplier and multiple retailers have been extensively studied in the relevant literature, see e.g., \cite{Ty99,Be05,Ya06} and \cite{Wu12} among others. The absence of upstream competition in these models is a simplifying factor that concentrates the analysis on the impact of the supplier's decisions to the retail market. Our study considers the additional assumptions of in-house production capacities and demand uncertainty and aims to contribute in this stream of literature.\par
The study of such models is also theoretically justified. While one may think, monopolies are rare in the contemporary congested economies, a more detailed reflection suggests the opposite. For instance, oligopolistic firms differentiate their products in such a degree that they monopolize specific market niches. Moreover, brand loyalty, increasing market protection, government policies, diversification and technology integration are only some of the prevalent economic factors that still lead to the formation of monopolies. Hence, monopoly pricing remains relevant and continues to attract research attention, see e.g., \cite{Ch04}.\par
Despite these arguments, it is self-evident from an economic perspective that upstream competition should be considered as an extension of the present model in future studies. This research direction will primarily investigate the suppliers' strategies to avoid straightforward price competition and hence, the trivial Bertrand equilibrium according to which they will sell at their marginal cost. Such an approach involves the study of product differentiation, favourable buy-back conditions and other more involved research \& development (R\&D) and marketing policies and as a result shifts the point of interest from pricing to strategic planning. The insights obtained by the present work can serve as a benchmark for comparison in such a study. \vspace{0.3cm}

\textbf{Uniform pricing:} We study a setting in which the supplier has ample quantities to cover any possible demand by the retailers (domestic manufacturers). The retailers only refer to the supplier if they exhaust their own production capacities and order quantities that are not comparable in size to the supplier's capacity (big international manufacturer). This model aims to capture the prevalent economic practice, where domestic manufacturers hedge their risks that are associated with demand fluctuations by reverting to the international \emph{spot} markets of the same product in cases of a stockout.\par
Yet, price discrimination is an interesting and direct extension of the present study. Price differentiation has been thoroughly studied as a mechanism to increase supply chain coordination and profits and recent studies corroborate the position that price differentiation can benefit the social welfare as well, see e.g., \cite{Wu18}. \par
However, there are certain arguments that still motivate the study of the uniform-price setting as employed in this paper. In addition to the fact, that price discrimination is prohibited by law in certain markets (see e.g., Robinson -- Patman Act), there are structural properties of contemporary markets that make the study of supply chains without price differentiation interesting, see e.g. \cite{Hu13}. Two notable instances are internet platforms and internationally operating suppliers as mentioned above. Manufacturers from domestic markets turn to such markets to procure additional quantities in the case that the short-term demand exceeds their own production capacities. This is done in an ad-hoc fashion (spot market) without significant prior communication between the manufacturer and the global (or internet operating) supplier and in particular, without the opportunity to develop a more elaborate transaction scheme (contract) between them. In essence, the globally operating supplier may ignore several characteristics of the market that he is procuring, such as the level of demand or the number of domestic retailers/manufacturers that operate in this local market. Finally, it is worth mentioning that recent advances provide evidence that uniform pricing is more profitable for the supplier than price discrimination for particular market structures, see \cite{Ma18}. \vspace{0.3cm}

\textbf{Price-only contracts:} Price-only contracts are still studied in the literature for two interconnected reasons: the first is their simplicity (in terms of implementation) and the second their relevance to the real economic practice, see e.g., \cite{La01} or \cite{Pe07}. More complex contracts often are costly or hard to implement and hence, the classic price-only contract still prevails in many vertical markets. However, there is also a third reason that makes the study of price-only contracts interesting in the particular setting of demand uncertainty. The supplier and the retailers may well develop more complex contracts that indeed mitigate demand uncertainty and supply chain risks. Yet, under the majority of the approaches that have been proposed and studied in the relevant literature (such as dynamic pricing, buy-back contracts etc.), a certain degree of uncertainty remains unaccounted for. This \emph{remaining uncertainty} cannot be dealt with and is precisely captured by the price-only contract. In this direction, \cite{Li17} show that uncertainty reduction may even be undesirable (harmful) for the involved parties and argue in favor of the use of price-only contracts. \vspace{0.3cm}

\textbf{Commitment on prices:} Invariably, in most settings, strategic decisions concerning production capacities are more binding than decisions concerning prices. This explains the extensive literature on models with capacity or quantity (instead of price) pre-commitment. However, while customary and intuitive to assume flexible prices, contemporary research suggests that the opposite may be true as well, see e.g., \cite{Kl10,Go11,Na13} (and references cited therein). In fact, \emph{price rigidity} or equivalently \emph{price stickiness} is an economic phenomenon that is evinced in the real economic practice to a much larger extent than intuition would suggest, see \cite{Kl09} and \cite{Dh09} among others. In addition, in a short term, one-period analysis, as the one that is conducted in the present paper, it is theoretically justified to assume constant prices. \par
Recent results provide a further argument in favor of price commitment by showing that wholesale price commitment can be beneficial in its own right and thus may be employed by the supplier even in the presence of price flexibility. As \cite{Wu12} note in their literature review, early wholesale price commitment is valuable for realizing downstream operational decisions before the investment is made, \cite{Gi03}, coordinating supply chain performance, \cite{De04}, or deterring the retailer from introducing a store brand, \cite{Gr10}.\par
Finally, changing or adjusting an unreasonable price is certainly possible in a multiple-period dynamic setting which can be thought of as a reasonable and interesting extension of the present one-period setting in which prices remain constant. Indeed, it would be relevant both from a theoretical and a practical perspective to adjust the present model to account for demand uncertainty that is distributed over multiple periods.

\subsection{Future work}\label{secfuture}
While the influence of the main model parameters has been thoroughly examined, there are still many directions for more exhaustive comparative statics. From a technical perspective, affinity of the inverse demand function may not be dropped if one still wants to express the supplier's payoff in terms of the MRL function. Hence, a generalization to models with non-affine inverse demand function is open. Finally, because the economic practice provides a huge variety of market structures, the extensions of the present model can be thought of in numerous ways. One immediate direction is the strategic determination of the production capacities $T$ (or $T_i$ if the retailers are asymmetric), which are assumed to be given exogenously (i.e., to be pre-specified) in the present model. Because production capacities are usually more rigid than prices, this should be done in a stage prior to the pricing decision of the supplier. Challenging the basic model assumptions -- e.g., that supplier's cost is higher than in-house production cost -- to account for different market structures, while retaining the basic analytical results of the present analysis (equilibrium strategies characterization) is another possible extension.

\section{Conclusions}
To study the strategic formation of the competing firms' cost in a modern Cournot oligopoly, we extended the classic model to a two-stage game. Oligopolists may produce limited capacities of a homogeneous good and refer to an external supplier for additional procurements. The supplier can cover any demand but sets a price prior to the retailers' orders without necessarily knowing the exact market parameters (retail demand). When the supplier is completely informed about the market demand, we derived the unique retailers' and supplier's equilibrium strategies as functions of the capacity levels, the  demand value and the supplier's production cost. While equilibrium strategies, payoffs of the retailers, and equilibrium payoff of the supplier are continuous and increasing in the demand level $\alpha$, the supplier's optimal pricing strategy may not be continuous and not even monotonic in $\alpha$. When the supplier is not informed of the retail demand, his belief is modeled by a continuous probability distribution function of the demand parameter. Under the mild assumption that the mean residual expectation of his belief about the actual demand is decreasing (DMRL property), we established existence and uniqueness of a Bayes-Nash equilibrium. Additionally, we characterized the supplier's optimal price as a fixed point of a translation of the MRL function, which enabled a tractable comparative statics and sensitivity analysis. \par
A comparison of the complete and incomplete information equilibrium outcomes in \Cref{secinef} indicated that under the supplier's optimal pricing policy, there is a considerable risk of no transaction between the supplier and the retailers, although such a transaction would have been beneficial for everyone. In that case, a reduced quantity will eventually be released to the consumers. Intuition on the reasons that lead the supplier to ask such a high price is gained through analytical and numerical considerations. If the retailers' production (or inventory) capacity increases, then (at equilibrium) the supplier's profit margin and wholesale price both decrease. If the supplier's cost increases, then his profit margin decreases, but the wholesale price that he asks increases. In this case, the retailers' orders to the supplier and hence the quantity sold to the consumers, both decrease. Finally, under the assumption that the retailers are symmetric, i.e., each retailer has production capacity $T$, we determined the market equilibrium for any number $n$ of them. Based on this extension, we studied the impact of an increasing number of retailers on the supplier's profit for various distributions.

\section*{Acknowledgments}
Stefanos Leonardos gratefully acknowledges support by a scholarship from the Alexander S. Onassis Public Benefit Foundation.

\appendix
\section{Appendix: General statements and proofs}\label{app1}
\subsection{Proofs of \texorpdfstring{\Cref{sec3}}{u}}

\begin{proof}[\textbf{Proof of \Cref{lem3.2}}]
For convenience in notation we will give the proof for $i=1$ and $j=2$ which results in no loss of generality due to symmetry of the retailers. Let $w\ge c$ and $s_2=\(t_2, q_2\) \in S^2$ with $Q_2=t_2+q_2$. Restricting ourselves firstly to strategies with $q_1=0$, we have that 
\begin{equation*}u^1\(\(t_1,0\),s_2\)=t_1\(\alpha-Q_2-t_1\) \end{equation*}
Let $t_1^*=\argmax_{0\le t_1 \le T} u^1\(\(t_1,0\), s_2 \)$. Since $\alpha-Q_i\ge 0, \, i=1,2$, we have that if $\alpha-2T\le Q_2$, then $t_1^*=\frac12\(\alpha-Q_2\)$ while if $\alpha-2T>Q_2$, then the maximum of $u^1\(\(t_1,0\),s_2\)$ is attained at the highest admissible (due to the production capacity constraint) value of $t_1$ and hence $t^*_1=T$. Similarly, for strategies with $t_1=T$ and $q_1>0$ we have that the payoff function of retailer $R_1$
\begin{align*}u^1\(\(T,q_1\),s_2\)&=Q_1\(\alpha-w-Q_2-Q_1\)+Tw\\&=T\(\alpha-Q_2-T\)+q_1\(\alpha-w-2T-Q_2-q_1\)\end{align*}
is maximized at $q^*_1=\dfrac12\(\alpha-w-2T-Q_2\)$ under the constraint that $q^*_1>0$ or equivalently $Q_2 < \alpha-w-2T$. If instead $q^*_1\le0$, then the maximum of $u^1\(\(T,q_1\),s_2\)$ is attained at the lowest admissible value of $q_1$ i.e. $q_1=0$. Since conditions $\alpha-2T\le Q_2$ and $Q_2<\alpha-w-2T$ cannot apply at the same time, the result obtains. 
\end{proof}

\begin{proof}[\textbf{Proof of \Cref{prop3.3full}}]
The proof proceeds by examining under what conditions a second stage equilibrium occurs as an intersection of parts $i,j = (1), (2), (3)$ of the best reply correspondences of the retailers.  \\
\emph{Case} $\Gamma_{11}$. In this case the best reply correspondences intersect in their parts denoted by (1). By rearranging part (1) of the best reply correspondence in Lemma \ref{lem3.2} and assuming that $R_i$ replies optimally to $R_j$ for $i,j=1,2$, we obtain that the total quantities released to the market under equilibrium will be given by $Q_1^*=Q_2^*=\frac{\alpha-w}{3}$ subject to the constraints $Q_2^*< \alpha-w-2T_1$ and $Q_1^*< \alpha-w-2T_2$. Substituting the values $Q^*_1$ and $Q^*_2$ in the constraints we find that these solutions are acceptable if $\frac{\alpha-w}{3}< \alpha-w-2T_1$ and $\frac{\alpha-w}{3}< \alpha-w-2T_2$ or equivalently if $\max{\left\{T_1, T_2\right\}} < \frac{\alpha-w}{3}$ which may be reduced to $T_1 <\dfrac{\alpha-w}{3}$ since $T_2\le T_1$ is assumed. Combining the above relations and decomposing $Q^*_i$ as in Lemma \ref{lem3.2} we obtain the first line of \Cref{tab:1} that settles case $\Gamma_{11}$. Similarly,\\
\emph{Case} $\Gamma_{21}$. Now, $Q^*_1=T_1$ and $Q^*_2=\frac{\alpha-w-T_1}{2}$ subject to $\alpha-w-2T_1\le \frac{\alpha-w-T_1}{2} < \alpha-2T_1$ and $0\le T_1 < \alpha-w-2T_2$. Solving the constraints, yields $\max{\left\{T_1+T_2+w, 3T_1-w\right\}}< \alpha \le 3T_1+w$. \\
\emph{Case} $\Gamma_{22}$. Now, $Q^*_i=T_i$, for $i=1,2$. These strategies are acceptable if $\alpha-w-2T_1\le T_2 < \alpha-2T_1$ and $\alpha-w-2T_2\le T_1 < \alpha-2T_2$ which gives $2T_1+T_2 < \alpha \le T_1+2T_2+w$.\\
\emph{Case} $\Gamma_{31}$. Now, $Q^*_1=\frac{\alpha-Q^*_2}{2}$ and $Q^*_2=\frac{\alpha-w-Q^*_1}{2}$. Hence, $Q^*_1=\frac{\alpha+w}{3}$ and $Q^*_2=\frac{\alpha-2w}{3}$ subject to $\alpha-2T_1\le \frac{\alpha-2w}{3}$ and $0\le \frac{\alpha+w}{3}< \alpha-w-2T_2$ or equivalently $3T_2+2w< \alpha \le 3T_1-w$.\\
\emph{Case} $\Gamma_{32}$. It is easy to see that $Q^*_1=\frac{\alpha-T_2}{2}$ and $Q^*_2=T_2$ subject to $\alpha-2T_1\le T_2$ and $\alpha-w-2T_2\le \frac{\alpha-T_2}{2} < \alpha-2T_2$, which yields $3T_2 < \alpha \le 3T_2+2w$. Finally,\\
\emph{Case} $\Gamma_{33}$. $Q^*_1=Q^*_2=\frac{\alpha}{3}$ subject to $\alpha-2T_1\le \frac{\alpha}{3}$ and $\alpha-2T_2\le \frac{\alpha}{3}$ or equivalently $\alpha \le 3T_2$.   
\end{proof}

\begin{proof}[\textbf{Proof of \Cref{supay}}]
By \Cref{tab:1}
\begin{align*}q^*_{31}\(w\)&=\frac13 \(\alpha-2w-3T_2\), &&\hspace{-10pt}\text{ if } 3T_2+2w < \alpha \le 3T_1-w \\
q^*_{21}\(w\)&=\frac12 \(\alpha-w-T_1-2T_2\), &&\hspace{-10pt}\text{ if} \max{\{3T_1-w, T_1+2T_2+w\}} < \alpha \le 3T_1+w \\
q^*_{11}\(w\)&=\frac23 \(\alpha-w-\frac32T_1-\frac32T_2\), &&\hspace{-10pt} \text{ if } 3T_1+w < \alpha \\ 
q^*_{ij}\(w\)&=0, &&\hspace{-10pt} \text{ else }\end{align*}
and since $w=r+c$, with $c>0$ being a constant and $r \geq 0$ being the strategic variable, the payoff function of the supplier may be written as
\[u^s\(r\mid \alpha\) = r\cdot \begin{cases}
q^*_{31}\(r+c\), & 3T_2+2c+2r< \alpha \le 3T_1-c-r \\
q^*_{21}\(r+c\), & \max{\left\{3T_1-c-r, T_1+2T_2+c+r\right\}}< \alpha \le 3T_1+c+r \\
q^*_{11}\(r+c\), & 3T_1+c+r < \alpha \\ 
0, & \text{ else} \end{cases}\]
for $r\ge 0$, assuming that a subgame perfect equilibrium is played in the second stage. Re-arranging the conditions in the last column and taking into account the continuity of $u^s\(r\mid \alpha\)$ at the cutting points of the $\alpha$-intervals and the non-negativity constraint for $r$, the claim follows.
\end{proof}

\begin{proof}[\textbf{Proof of \Cref{prop3.4full}}]
The optimal values $r^*$ are denoted by $r^*_{ij}, j\le i\in \{1,2,3\}$, according to the equilibrium that is played in the second stage.\\[0.2cm]
\textbf{Case A.} Let $0<c\le D$. Then, $3T_2+2c \le T_1+2T_2+c \le 2T_1+T_2 \le 3T_1-c \le 3T_1+c$
and hence for the different values of $\alpha$ we have \\
\textbf{A1.} $0\le \alpha \le 3T_2+2c$. For all $r \in \mathbb R_+$ we have that $u^s\(r\mid \alpha\) \equiv 0$ and hence $r^*=r \in \mathbb R_+$. \\
\textbf{A2.} $3T_2+2c \le \alpha \le 2T_1+T_2$. The supplier's payoff function is given by
\[u^s\(r\mid \alpha\)=r\cdot \begin{cases}q^*_{31}\(r+c\), & 0\le r \le \frac12\(\alpha-2c-3T_2\) \\0, & \text{ else}\end{cases}\]
and hence, as a quadratic polynomial in $r$, it's first part is maximized at $r^*=r^*_{31}$.\\
\textbf{A3.} $2T_1+T_2\le \alpha \le 3T_1-c$. Now 
\[u^s\(r\mid \alpha\)=r\cdot \begin{cases}q^*_{31}\(r+c\), & 0\le r \le 3T_1-c-\alpha \\q^*_{21}\(r+c\), & 3T_1-c-\alpha \le r \le \alpha-c-T_1-2T_2 \\ 0, & \text{ else}\end{cases}\]
As quadratic polynomials in $r$, the first part is maximized at $r^*_{31}$ and the second at $r^*_{21}$. It is easy to see that $0\le r^*_{31}\le r^*_{21}\le \alpha-c-T_1-2T_2$. Hence, in order to determine the maximum of $u^s\(r\mid \alpha\)$ with respect to $r$ we distinguish three sub-cases. Let $S:=T_1+T_2$, then \par
\textbf{A3.1.} $0 \le r^*_{31} \le r^*_{21}\le 3T_1-c-\alpha$. Then the overall maximum of $u^s\(r\mid \alpha\)$ is attained at $r^*_{31}$. The inequality $r^*_{21}\le 3T_1-c-\alpha $ holds iff $\frac12\(\alpha-c-T_1-2T_2\)\le 3T_1-c-\alpha \Leftrightarrow \alpha \le \frac32S+\frac56D-\frac13c$.\par
\textbf{A3.2.} $3T_1-c-\alpha \le r^*_{31}\le r^*_{21}$. Then the overall maximum of $u^s\(r\mid \alpha\)$ is attained at $r^*_{21}$. The inequality $3T_1-c-\alpha \le r^*_{31}$ holds iff $ 3T_1-c-\alpha \le \frac14\(\alpha-2c-3T_2\) \Leftrightarrow \alpha \ge \frac32S+\frac9{10}D-\frac25c$.\par
\textbf{A3.3.} $0 \le r^*_{31} \le 3T_1-c-\alpha \le r^*_{21}$. In this case we need to compare the payoffs $u^s\(r^*_{31}\mid \alpha\)$ and $u^s\(r^*_{21}\mid \alpha\)$. The overall maximum is attained at $r^*_{31}$ iff $u^s\(r^*_{21}\mid \alpha\) \le u^s\(r^*_{31}\mid \alpha\) \Leftrightarrow \frac18\(\alpha-c-T_1-2T_2\)^2 \le \frac1{24}(\alpha-2c-3T_2)^2$. Both terms $\(\alpha-c-T_1-2T_2\)$ and $\(\alpha-2c-3T_2\)$ are non-negative since $\alpha \ge 2T_1+T_2$ and $c \le T_1-T_2$ hold by assumption. Hence, we may take the square root of both sides to obtain that $u^s\(r^*_{21}\mid \alpha\) \le u^s\(r^*_{31}\mid \alpha\)$ if and only if $\(\sqrt{3}-1\)\alpha \le \sqrt{3}T_1+\(2\sqrt{3}-3\)T_2+\(\sqrt{3}-2\)c$ which is in turn equivalent to $\alpha \le \frac32S+\frac{\sqrt{3}}{2}D-\frac{\sqrt{3}-1}{2}c$. Now, it is straightforward to check that the following ordering is equivalent to $c\leq D$, which is true by the defining condition of Case A. 
\[2T_1+T_2 \leq \frac32S+\frac56D-\frac13c \leq \frac32S+\frac{\sqrt{3}}{2}D-\frac{\sqrt{3}-1}{2}c \le \frac32S+\frac9{10}D-\frac25c \le 3T_1-c\]
Hence, by the previous discussion and after observing that $\frac32S+\frac{\sqrt{3}}{2}D-\frac{\sqrt{3}-1}{2}c = 3T_2+\Delta-\frac{\sqrt{3}-1}{2}c$, we conclude that the optimal solution $r^*$ in this case is given by 
\[r^*=\begin{cases}r^*_{31}=\frac14\(\alpha-2c-3T_2\), & \text{ if }\,\, 2T_1+T_2\le \alpha \le  3T_2+\Delta-\frac{\sqrt{3}-1}{2}c\\[0.2cm]
r^*_{21}=\frac12\(\alpha-c-T_1-2T_2\), & \text{ if } \,\,\, 3T_2+\Delta-\frac{\sqrt{3}-1}{2}c \le \alpha < 3T_1-c\end{cases}\]
\textbf{A4.} $3T_1-c \le \alpha \le 3T_1+c$. Now, the supplier's payoff function is given by
\[u^s\(r\mid \alpha\)=r\cdot \begin{cases}q^*_{21}\(r+c\), & 0\le r \le \alpha-c-T_1-2T_2 \\0, & \text{ else}\end{cases}\]
and hence, as a quadratic polynomial in $r$, it is maximized at $r^*=r^*_{21}$.\\
\textbf{A5.} $3T_1+c\le \alpha$. Now 
\[u^s\(r\mid \alpha\)=r\cdot \begin{cases}q^*_{11}\(r+c\), & 0\le r \le \alpha-c-3T_1 \\q^*_{21}\(r+c\), & \alpha-c-3T_1 \le r \le \alpha-c-T_1-2T_2 \\ 0, & \text{ else} \end{cases}\]
Now, the first part is maximized at $r^*_{11}$ and the second at $r^*_{21}$. Again, one checks easily that $0 < r^*_{11} < r^*_{21} < \alpha-c-T_1-2T_2$. As in Case A3, in order to determine the maximum of $u^s\(r\mid \alpha\)$ with respect to $r$ we distinguish three sub-cases.\par
\textbf{A5.1.} $0< r^*_{11}< r^*_{21}\le \alpha-c-3T_1$. Then the overall maximum of $u^s\(r\mid \alpha\)$ is attained at $r^*_{11}$. The inequality $r^*_{21}\le \alpha-c-3T_1$ holds iff $\frac12\(\alpha-c-T_1-2T_2\)\le \alpha-c-3T_1 \Leftrightarrow \alpha \geq \frac32 S+\frac72 D+c$.\par
\textbf{A5.2.} $\alpha-c-3T_1\le r^*_{11}< r^*_{21}$. Then the overall maximum of $u^s\(r\mid \alpha\)$ is attained at $r^*_{21}$. The inequality $\alpha-c-3T_1\le r^*_{11}$ holds iff $\alpha-c-3T_1\le \frac12\(\alpha-c-\frac32T_1-\frac32T_2\) \Leftrightarrow \alpha \leq \frac32 S+3D+c$.\par
\textbf{A5.3.} $0<r^*_{11} < \alpha-c-3T_1 < r^*_{21}$. In this case we need to compare the payoffs $u^s\(r^*_{11}\mid \alpha\)$ and $u^s\(r^*_{21}\mid \alpha\)$. The overall maximum is attained at $r^*_{11}$ if and only if $u^s\(r^*_{21}\mid \alpha\) \le u^s\(r^*_{11}\mid \alpha\)$ or equivalently if $\frac18(\alpha-c-T_1-2T_2)^2 \le \frac16\(\alpha-c-\frac32T_1-\frac32T_2\)^2$. Both terms $\(\alpha-c-T_1-2T_2\)$ and $\(\alpha-c-\frac32T_1-\frac32T_2\)$ are positive since $\alpha \ge 3T_1+c$ holds by assumption. Hence, we may take the square root of both sides to obtain that $u^s\(r^*_{21}\mid \alpha\) \le u^s\(r^*_{11}\mid \alpha\)$ if and only if $\(3-\sqrt{3}\)T_1+\(3-2\sqrt{3}\)T_2+\(2-\sqrt{3}\)c \le \(2-\sqrt{3}\)\alpha$ which is in turn equivalent to $\frac32 S+ \(\frac32 + \sqrt{3}\)D +c \le \alpha$.\par
Now, since $\frac32 S+\(\frac32+\sqrt{3}\)D+c = 3T_2+2\Delta+c$ and $3T_1+c \le \frac32 S+3D+c \le \frac32 S+\(\frac32+\sqrt{3}\)D+c \le \frac32 S+\frac72 D+c$, we conclude that the optimal solution $r^*$ in this case is given by 
\[r^*=\begin{cases}r^*_{21}=\frac12\(\alpha-c-T_1-2T_2\), & \text{ if }\,\, 3T_1+c\le \alpha \le  3T_2+2\Delta+c\\[0.2cm]
r^*_{11}=\frac12\(\alpha-c-\frac32T_1-\frac32T_2\), & \text{ if } \,\,\, 3T_2+2\Delta+c \le \alpha \end{cases}\]
This concludes case \textbf{A}.\\[0.2cm]
\textbf{Case B.} Let $0\le D<c$. Then, $3T_1-c< 2T_1+T_2 < T_1+2T_2+c < \min \(3T_1+c,3T_2+2c\)$ and hence for the different values of $\alpha$ we have \\
\textbf{B1.} $0\le \alpha \le T_1+2T_2+c$. Now $u^s\(r\mid \alpha\) \equiv 0$ for all $r \in \mathbb R_+$ and hence $r^*=r \in \mathbb R_+$. \\
\textbf{B2.} $T_1+2T_2+c \le \alpha \le 3T_1+c$. Now, the supplier's payoff function is given by
\[u^s\(r\mid \alpha\)=r\cdot \begin{cases}q^*_{21}\(r+c\), & 0\le r \le \alpha-T_1-2T_2-c \\0, & \text{ else}\end{cases}\]
and hence as a quadratic polynomial in $r$, it is maximized at $r^*=r^*_{21}$.\\
\textbf{B3.} $3T_1+c\le \alpha$. This case is identical to A5 above. 
Collecting the results about the optimal strategy of the supplier in all sub-cases of \textbf{A} and \textbf{B}, we obtain the claim of Proposition \ref{prop3.4full}.
\end{proof}

\subsection{Proofs of \texorpdfstring{\Cref{sec4}}{l}}

\begin{lemma}\label{lem diffa}
The supplier's payoff function $\us$ is continuously differentiable on $\(0,r_H\)$ and
\begin{equation}\label{10} \fdr=\frac23\int_{3T+c+r}^{\infty}\(1-F\(x\)\)dx-\frac23r\(1-F\(3T+c+r\)\). \end{equation} 
\end{lemma}
\begin{proof} First observe that since $\(\alpha-3T-c-r\)^{+}$ is non-negative, $\ex(\alpha-3T-c-r)^{+} = \int_{0}^{\infty}P((\alpha-3T-c-r)^{+} >y) dy$, see e.g.\cite{Bi86} and hence, by a simple change of variable, $\ex(\alpha-3T-c-r)^{+} = \int_{3T+c+r}^{\infty}\(1-F\(x\)\)dx$ which implies that
\begin{equation}\label{9} \us=\frac23r\int_{3T+c+r}^{\infty}\(1-F\(x\)\)dx \;\; \mbox{for $0 \le r<r_H$}. \end{equation}
Hence, to prove the assertion of the Lemma, it suffices to show that $\frac{\mathop{d}}{\mathop{dr}}\ex\(\alpha-3T-c-r\)^{+} = - (1-F(3T+c+r))$. Then, equation \eqref{10} as well as continuity are implied. So, let 
\begin{equation}\label{11}  K_{h}\(\alpha\):=-\frac1h\,\left[\(\alpha-3T-c-r-h\)^+-\(\alpha-3T-c-r\)^+\right] \end{equation} and take $h>0$. Then, \[K_h\(\alpha\)= \mathbf{1}_{\{\alpha>3T+c+r+h\}} + \dfrac{\alpha-3T-c-r}{h}\mathbf{1}_{\{3T+c+r< \alpha \le 3T+c+r+h\}}\] and therefore $\lim_{h\to 0+}K_{h}\(\alpha\)=\mathbf{1}_{\{\alpha>3T+c+r\}}$. Since $0\le K_h\(\alpha\) \le 1$ for all $\alpha$, the dominated convergence theorem implies that $\lim_{h\to0+}\ex\(K_h\(\alpha\)\) = P\(\alpha>3T+c+r\)$. In a similar fashion, one may show that $\lim_{h\to0-}\ex\(K_h\(\alpha\)\) = P\(\alpha \geq 3T+c+r\)$. Since the distribution of $\alpha$ is non-atomic, $P\(\alpha>3T+c+r\) = P\(\alpha \geq 3T+c+r\)$ and hence, $\lim_{h\to0}\ex\(K_h\(\alpha\)\) = 1-F\(3T+c+r\)$. By \cref{11}, $\lim_{h\to0}\ex\(K_h\(\alpha\)\) = -\frac{\mathop{d}}{\mathop{dr}}\ex\(\alpha-3T-c-r\)^{+}$, which concludes the proof.
\end{proof}

\begin{proof}[\textbf{Proof of \Cref{lem diff}}]
The formulas for $\us$ and $\fdr$ are immediate using \cref{9}, \cref{10}, and \cref{imrl}. We remark that since $\ex\(\alpha-3T-c-r\)^{+} = \m\(3T+c+r\)\(1-F\(3T+c+r\)\)$, one may be tempted to use the product rule to derive its derivative and hence show that $\us$ is differentiable. However, the product rule does not apply, since the two terms in this expression of $\ex\(\alpha-3T-c-r\)^{+}$ may both be non-differentiable, even if $\alpha$ has a density and its support is connected (e.g. consider the point $r_L$ in case $r_L>0$). Finally, if $r \in \(0,r_H\)$, then $1-F\(3T+c+r\)$ is positive. Hence, equation \eqref{der} implies that the critical points $r^*$ of $\us$ (if any) satisfy $r^*=\m\(r^*+3T+c\)$.   \end{proof} 

\begin{proof}[\textbf{Proof of \Cref{thm4.1}}]
Due to Lemma \ref{lem diff}, if a non-zero optimal response of the supplier exists at equilibrium, it will be a critical point of $\us$, i.e. it will satisfy \eqref{14}. It is easy to see that such a response always exists when the support of $\alpha$ is bounded, i.e. when $\alpha_H < \infty$. However, this is not the case when $\alpha_H = \infty$. So, let us determine conditions under which such a critical point exists, is unique and corresponds to a global maximum of the supplier's payoff function. To this end, we study the first term of \eqref{der}, namely $g(r):=\m\(r+3T+c\) - r$. \par 
Clearly, $g\(r\)$ is continuous on $\(0,r_H\)$. We first show that $\lim_{r \to 0+} g(r) > 0$ by considering cases. If $r_L > 0$, then for $0 < r < r_L$, $\m\(r+3T+c\) = \ex(\alpha) -r -3T -c$. Hence, $\lim_{r \to 0+} g(r) = \ex(\alpha) -3T -c> r_L >0$. If $r_L \leq 0$, then we use Proposition 1f of \cite{Ha81}, according to which $\m\(t\)\ge \(\ex\(\alpha\)-t\)^+$ with equality if and only if $F\(t\)=0$ or $F\(t\)=1$. So, take first $r_L < 0$, which implies $\alpha_L < 3T+c < \alpha_H$. Hence, $0< F(3T+c) < 1$ and (by Proposition 1f) $\m \(3T+c\) > \ex\(\alpha -3T-c\)^+ \geq 0$. Hence, $\lim_{r \to 0+} g\(r\) > 0$ if $r_L < 0$. Finally, if $r_L = 0$, then $\m \(3T+c\) =\m\(\alpha_L\)=\ex\(\alpha\)-\alpha_L >0$, which again implies that $\lim_{r \to 0+} g\(r\) > 0$. \par
We then examine the behavior of $g\(r\)$ near $r_H$. If $\alpha_H<+\infty$, then $\lim_{r \to r_H-} g\(r\) = -r_H < 0$ and by the intermediate value theorem an $r^* \in \(0,r_H\)$ exists such that $g\(r^*\)=0$. For $\alpha_H<+\infty$, we also notice that to get uniqueness of the critical point $r^*$, it suffices to assume that the Mean Residual Lifetime (MRL) of the distribution of $\alpha$ is decreasing\footnote{We use the terms ``decreasing'' in the sense of ``non-increasing'' (i.e. flat spots are permitted), as they do in the pertinent literature, where this use of the term has been established.}, in short that $F$ has the DMRL property. On the other hand, if $\alpha_H=+\infty$, then the limiting behavior of $\m\(r\)$ as $r$ increases to infinity may vary, see \cite{Br03}, and an optimal solution may not exist. But, if we assume as before that $\m\(r\)$ is decreasing, then $g\(r\)$ will eventually become negative and stay negative as $r$ increases and hence, existence along with uniqueness of an $r^*$ such that $g\(r^*\)=0$ is again established.\par
Now, $1-F\(3T+c\) >0$ since $3T+c < \alpha_H$ and hence $\lim_{r \to 0+} \fdr > 0$, i.e. $\us$ starts increasing on $\(0, r_H\)$. Assuming that $F$ has the DMRL property, the first term of \eqref{der} is negative in a neighborhood of $r_H$ while the second term goes to 0 from positive values. Hence, $\fdr < 0$ in a neighborhood of $r_H$, i.e $\us$ is decreasing as $r$ approaches $r_H$.
Clearly, for $\epsilon$ sufficiently small, $\us$ will take a maximum in the interior of the interval $[\epsilon, r_H -\epsilon]$ if $r_H < \infty$ or a maximum in the interior of the interval $[\epsilon, \infty)$ if $r_H = \infty$. Since $\us$ is differentiable, the maximum will be attained at a critical point of $\us$, i.e. at the unique $r^*$ given implicitly by \eqref{14}.\par
\Cref{14} actually characterizes $r^*$ as the fixed point of a translation of the MRL function $\m\(\cdot\)$, namely of $\m\(\cdot +3T+c\)$. Its evaluation sometimes has to be numeric, but in one interesting case it may be evaluated explicitly: If $r_L >0$, then $\alpha -3T-c > 0$ for all $\alpha$, which implies that $\ex(\alpha) -3T-c > 0$. Then, if $\frac12\(\ex\(\alpha\) -3T-c\) \leq r_L,$
we get $r^* = \frac12\(\ex(\alpha) -3T-c\)$. To see this, notice that the previous equation is equivalent to $\m\(\alpha_L\) \le r_L$, i.e. to $\m\(r_L+3T+c\)\le r_L$. Then, by the DMRL property $\m\(r+3T+c\)<r$ for all $r>r_L$. This implies that $r^*\le r_L$ or equivalently that $r^*+3T+c \le \alpha_L$. In this case $\m\(r^*+3T+c\)=\ex\(\alpha\)-\(r^*+3T+c\)$ and hence $r^*$ will be given explicitly by $r^*=\frac12\(\ex\(\alpha\)-3T-c\)$. Intuitively, this special case occurs under the conditions that (a) the lower bound of the demand $\alpha_L$ exceeds the particular threshold $3T+c$, (i.e. $\alpha_L > 3T+c$ \, or \, $r_L > 0$), and (b) the expected excess of $\alpha$ over its lower bound $\alpha_L$ is at most equal to the excess of $\alpha_L$ over $3T+c$ (i.e. $\ex\(\alpha\) - \alpha_L \leq \alpha_L -3T-c = r_L$). Of course, since $\ex\(\alpha\) - \alpha_L >0$, condition (b) suffices. In that case, compare with the optimal $r^*$ of the complete information case (Proposition \ref{prop3.4}). \par
Finally, if $r_L >0$ and $\frac12\(\ex\(\alpha\) -3T-c\) > r_L$ (i.e, if $\ex\(\alpha\) > r_L +\alpha_L$) we get that $r^* > r_L$, for if $r^* \leq r_L$, then $\m\(r^*+3T+c\)
\ge \m \(\alpha_L\)$, i.e. $r^* \ge \ex\(\alpha\) - \alpha_L$. The latter implies that then $r_L \ge \ex\(\alpha\) - \alpha_L$ which contradicts the assumption.
\end{proof}

\begin{proof}[\textbf{Proof of \Cref{prop7.1}}]
For $i\in N$, the best reply correspondence of retailer $R_i$ is given by Lemma \ref{lem3.2} if we replace $Q_j$ by $Q_{-i}$. Hence, we may simplify the proof by fixing $i \in N$ and distinguishing the following cases: \\[0.2cm]
\textbf{Case 1.} Let $0 \le \alpha \le \(n+1\)T$ and for $N \ni j\neq i$ let $t^*_j=\frac1{n+1}\alpha, q^*_j=0$. Then, $Q^*_{-i}=\frac{n-1}{n+1}\alpha\ge \alpha-2T$. Hence, by Lemma \ref{lem3.2}, $\operatorname{BR}^i\(Q^*_{-i}\)=\(t^*_i, q^*_i\)=\(\frac{\alpha}{n+1}, 0\)$.\\
\textbf{Case 2.} Let $\(n+1\)T< \alpha \le \(n+1\)T+w$ and for $N\ni j \neq i$ let $t^*_j=T, q^*_j=0$. Then $Q^*_{-i}=\(n-1\)T$ with $\alpha-w-2T\le Q^*_{-i}<\alpha-2T$. Hence, by Lemma \ref{lem3.2}, $\operatorname{BR}^i\(Q^*_{-i}\)=\(t^*_i, q^*_i\)=\(T, 0\)$.\\
\textbf{Case 3.} Let $\(n+1\)T+w<\alpha$ and for $N\ni j\neq i$ let $t^*_j=T, q^*_j=\frac1{n+1}\(\alpha-\(n+1\)T-w\)$. Then $Q^*_{-i}=\frac{n-1}{n+1}\(\alpha-w\)<\alpha-w-2T$. As above $\operatorname{BR}^i\(Q^*_{-i}\)=\(t^*_i, q^*_i\)=\(T, \frac1{n+1}\(\alpha-\(n+1\)T-w\)\)$.
Summing up, we obtain the equilibrium strategies as given in Proposition \ref{prop7.1}.
\end{proof}

\subsection{Proofs of \texorpdfstring{\Cref{secinef}}{dua}}

\begin{proof}[\textbf{Proof of \Cref{lemprob}}]
Using $S$ for the support of $F$, we get $U = \left\{\alpha\mid \alpha \in S  \mbox{ and }  \alpha > 3T +c\right\}$ and $V = \{\alpha\mid \alpha \in S \; \mbox{and} \; \alpha \leq r^* +3T +c\}$. So, $V \cap U = \{\alpha \mid \alpha \in S \; \mbox{and} \; 3T +c < \alpha \le r^* +3T +c\}$, and therefore $P\(V \cap U\) = F\(r^* +3T +c\) - F\(3T +c\)$ and \[P\(V \mid U\) = \frac{F\(r^* +3T +c\) - F\(3T +c\)} {1-F\(3T +c\)}\] \par
Since, by assumption $\ex\(\alpha\) - \alpha_L > r_L$, \Cref{thm4.1} implies that the supplier will sell at $r^*\in \(r_L^+, r_H\)$ in equilibrium. So, there are two cases, either (a) $r_L > 0$ or (b) $r_L \leq 0$. In case (a), we get $3T + c < \alpha_L < r^* +3T +c <\alpha_H$, hence $V \cap U = \left[\alpha_L, r^* +3T +c\right]\cap S$. In case (b), $\alpha_L \leq 3T + c < r^* +3T +c <\alpha_H$, hence $V \cap U = \(3T +c, r^* +3T +c\right]\cap S$ (notice that for $\alpha \leq 3T +c$ no transaction would have taken place under complete information also). Hence, $P\(V \cap U\) = F\(r^* +3T +c\) - F\(3T +c\)$ which concludes the proof.
\end{proof}

\begin{proof}[\textbf{Proof of \Cref{thmbound}}]
Firstly, we express the distribution function $F$ in terms of the MRL function, e.g. see \cite{Gu88}, to get
\begin{equation} \label{inv} 1-F\(t\)=\frac{\ex\(\alpha\)}{\m\(t\)}\exp{\left\{-\int_{0}^{t}\frac{1}{\m\(u\)}\mathop{du}\right\}} \; \; \mbox{for}\; \, 0\le t < \alpha_H\end{equation}
We will use the DMRL property, cf. \eqref{16}, to derive an upper bound for $P\(V \mid U\)$. To this end, we use \eqref {inv} first for $t=r^*+3T+c$, then for $t=3T+c$, and then, dividing the two equations (division by 0 is no danger, see the discussion preceding Lemma \ref{lemprob}), we get
\begin{equation*}1-F\(r^*+3T+c\)=\(1-F\(3T+c\)\)\frac{\m\(3T+c\)}{\m\(r^*+3T+c\)}\exp{\left\{-\int_{3T+c}^{r^*+3T+c}\frac{1}{\m\(u\)}du\right\}} \, .\end{equation*}
Hence, since $P\(V \cap U\)=1-F\(3T+c\)-\(1-F\(r^*+3T+c\)\)$, we have that \[P\(V\cap U\)=\(1-F\(3T+c\)\)\(1-\frac{\m\(3T+c\)}{\m\(r^*+3T+c\)}\exp{\left\{-\int_{3T+c}^{r^*+3T+c}\frac{1}{\m\(u\)}du\right\}}\),\]
which shows that $P\(V \mid U\) = 1-\frac{\m\(3T+c\)}{\m\(r^*+3T+c\)}\exp{\left\{-\int_{3T+c}^{r^*+3T+c}\frac{1}{\m\(u\)}du\right\}}.$
By inequality \eqref{16} for $3T +c \leq u \leq r^* +3T +c$ and the monotonicity of the exponential function
\begin{equation*}\exp{\left\{-\frac{1}{\m\(r^*+3T+c\)}\int_{3T+c}^{r^*+3T+c}du\right\}}\le \exp{\left\{-\int_{3T+c}^{r^*+3T+c}\frac{1}{\m\(u\)}du\right\}} \, . \end{equation*}
Using the fact that $r^*=\m\(r^*+3T+c\)$, the last inequality becomes $\exp{\left\{-1\right\}}\le \exp{\left\{-\int_{3T+c}^{r^*+3T+c}\frac{1}{\m\(u\)}du\right\}}$. Substituting in $P\(V\mid U\)$ we derive the following upper bound for $P\(V \mid U\)$
\[P\(V \mid U\) \le 1-\frac{\m\(3T+c\)}{\m\(r^*+3T+c\)} e^{-1}\] Since $\m\(r^*+3T+c\)\le \m\(3T+c\)$ by the DMRL property, the upper bound may be relaxed to $P(V \mid U)\le 1-e^{-1}$. To see that this bound is indeed tight, let $\alpha \sim \exp\(\lambda\)$, i.e. $f\(\alpha\)=\lambda e^{-\lambda \alpha} 1_{\{0 \leq \alpha < \infty\}}$, with $\lambda >0$ and let $T\ge 0$. Since $\m\(t\)= \dfrac{1}{\lambda}-t\cdot \mathbf{1}_{\{t\le 0\}}$, $F$ is DMRL. By \Cref{thm4.1}, the optimal strategy $r^*$ of the supplier is independent of $T,c$ and is given by $r^*=\frac1\lambda$. The conditional probability of no transaction $P\(V \mid U\)$ is also independent of $T,c$ and equal to $1-e^{-1}$ due to the memoryless property of the exponential distribution (i.e. $P\(\alpha>s+t\mid \alpha>t\)=P\(\alpha>s\)$). Indeed,  
\begin{align*}P\(V \mid U\)&= P\(\alpha \leq r^*+3T+c \mid \alpha> 3T+c\)\\&=1-P\(\alpha> r^*+3T+c \mid \alpha> 3T+c\)\\&=1-P\(\alpha>r^*\)=F\(1/\lambda\)=1-e^{-1}\end{align*}
implying that the inequality in \eqref{bound} is tight over all DMRL distributions. 
\end{proof}

\setstretch{1}\fontsize{9}{11}\selectfont

\end{document}